\documentclass{amsart}

\usepackage{amsmath}
\usepackage{amsfonts}
\usepackage{amssymb,enumerate}
\usepackage{amsthm}
\usepackage{fancyhdr}
\usepackage{hyperref}

\newtheorem{lem}{Lemma}
\newtheorem{cor}[lem]{Corollary}
\newtheorem{prop}[lem]{Proposition}
\newtheorem{thm}[lem]{Theorem}

\newtheorem{con}[lem]{Conjecture}
\theoremstyle{definition}
\newtheorem{defin}[lem]{Definition}

\newtheorem{rem}[lem]{Remark}
\newtheorem{ex}[lem]{Example}
\newtheorem{discuss}[lem]{Discussion}

\numberwithin{lem}{section}

\newcommand{\NN}{\mathbb{N}_+}
\newcommand{\Sw}{\mathcal{S}}
\newcommand{\Tw}{\mathcal{T}}
\newcommand{\N}{\mathcal{N}}
\newcommand{\I}[2]{I^{\langle #1,#2\rangle}}
\newcommand{\Var}[2]{\mathrm{Var}^{\langle s,t\rangle}}
\newcommand{\It}[2]{\widetilde{I}^{\langle #1,#2\rangle}}
\renewcommand{\P}[2]{P^{\langle #1,#2\rangle}}

\newcommand{\sno}{\|_{\langle s \rangle}}
\newcommand{\nno}{\|_{\langle n \rangle}}

\newcommand{\sw}{\mathrm{sw}}
\newcommand{\di}{\mathrm{d}}
\def\lt{\mathrm{LT}}
\def\rad{\mathrm{rad}}
\def\Scroll{\mathrm{Scr}}

\title{$s$-Hankel hypermatrices and $2\times 2$ determinantal ideals}

\author{Alessio Sammartano}

\address{Department of Mathematics, Purdue University
150 N. University Street, West Lafayette, IN 47907, USA}

\email{asammart@purdue.edu} 

\date{}


\begin{document}

\begin{abstract}
We introduce the concept of $s$-Hankel hypermatrix, which generalizes both Hankel matrices and generic hypermatrices.
We study two determinantal ideals associated to an $s$-Hankel hypermatrix:
the ideal $\I{s}{t}$ generated by certain $2 \times 2$ slice minors, and the ideal $\It{s}{t}$ generated by certain $2 \times 2$ generalized minors.
 We describe the structure of these two ideals, with particular attention to the primary decomposition of $\I{s}{t}$, and provide the explicit list of minimal primes for large values of $s$.
Finally we give some geometrical interpretations and generalise a theorem of J. Watanabe.
\end{abstract}

\maketitle

\section*{Introduction}
The study of determinantal ideals is a central area of research in commutative algebra.
One of the basic results in this theory is due to Eagon and Hochster, 
who proved in \cite{EH} that the ideal generated by the $r\times r$ minors of a generic matrix (i.e., a matrix whose entries are distinct variables of a polynomial ring) is prime.
An analogous result was later proved by J. Watanabe in \cite{Wa} for Hankel matrices.
Recall that an $r_1 \times r_2 $ matrix is called Hankel if the $(a_1, a_2)$-entry   is a variable which only depends on the sum $a_1+a_2$, with $1 \leq a_i \leq r_i$ (cf. \cite{Co}).

This study has been extended to ideals generated by minors of hypermatrices,
 due to the interesting connections with tensors and  projective varieties (cf. \cite{Ber}, \cite{Gr}, \cite{Ha}) and  algebraic statistics (cf. \cite{ST}).
In this context,
the ideals treated so far are mainly generated by $2 \times 2$ minors of a generic hypermatrix.

Motivated by the parallelism between generic matrices and Hankel matrices,
 we introduce in this paper the definition of an $s$-Hankel hypermatrix.
If $M$ is an $r_1 \times r_2 \times \cdots \times r_n$ hypermatrix and $s\leq n$ is a positive integer, we say that
$M$ is $s$-Hankel if the $(a_1, a_2, \ldots, a_n)$-entry is a variable which only depends on the sum $\sum_{i=1}^s a_i$ and the $(n-s)$-tuple $(a_{s+1},\ldots, a_n)$.
 This concept  generalizes  several classical objects:
 for instance, if $s=1$ 
 then
$M$ is a generic hypermatrix, whereas
if $n=s=2$ then $M$ is a Hankel matrix.
When $n=s\geq 3$ we obtain a hypermatrix  whose $a$-entry only depends on the sum $\sum_{i=1}^n a_i$: 
such  objects are sometimes called ``Hankel tensors'' and appear in various areas of mathematics
(e.g.  \cite{BB}, \cite{BBL}, \cite{LT}, \cite{LT2}, \cite{PDV}).
Finally, for $s=n-1$ we obtain a class of hypermatrices related to certain rational normal scrolls described in Section \ref{secGeometry}.
 
Our main task is to study two classes of determinantal ideals associated to this hypermatrix and to this aim
we introduce another positive integer $t \leq n$.
We consider the ideals $\I{s}{t}$ and $\It{s}{t}$ which are generated by specific $2 \times 2$  minors of $M$ described in detail in  Definition \ref{definIdeals}.
The motivation for using $t$,
instead of just taking all minors choosing $t=n$,
 is to consider at once a wider class of ideals
and establish further connections with  existing literature.
 An  example arises from algebraic statistics:
for each value of $t\in [n]$, $\I{1}{t}$  corresponds to a class of
conditional independence statements   (cf.~\cite{AR} for $t=1$,~\cite{ST} for arbitrary $t$).
Another example is $\It{1}{t}$ being equal to the ideal of the Segre embedding $\mathbb{P}(V_1)\times \cdots \times \mathbb{P}(V_t) \times
\mathbb{P}(V_{t+1} \otimes \cdots \otimes V_n) \hookrightarrow \mathbb{P}(V_{1} \otimes \cdots \otimes V_n) $ (cf. \cite{Ha}).
 Besides, 
the parameter $t$ plays a key role in establishing a further connection with projective geometry in   Section~\ref{secGeometry}.

The content of the paper is the following.
We first describe the structure of the ideal $\It{s}{t}$ 
by characterizing in Theorem \ref{thmReductGeneral} the binomials lying in it:
we derive in this way the primeness of $\It{s}{t}$ (cf. Proposition \ref{propPrimeSingle}) and its Hilbert function (cf. Corollary \ref{corHilbert}) extending thus the work of H\`a on the generic case in \cite{Ha}.
The ideal $\I{s}{t}$ is not prime, and we
characterize its minimal primes in Theorem \ref{thmMinimalPrimes}.
In the second part of the paper we focus on ideals with large values of $s$, 
which carry a tamer structure:
we are able to provide the explicit list of generators of the minimal primes
(cf. Theorems \ref{thmMinimalPrimesSN}, \ref{thmMinimalPrimes4}, \ref{thmMinimalPrimes3})
and the geometric interpretation 
(cf. Corollaries \ref{corCurve}, \ref{corScroll});
see also Remark \ref{remLargeValueS} and Conjecture \ref{conjPrimary}.
Finally, Theorem \ref{thmGeneralizationWatanabe} is a broad generalization of \cite[Theorem 1]{Wa} on Hankel matrices.

Most of our methods are inspired by the   paper \cite{ST} mentioned above, which deals with the generic  case.
We use in particular  the concept of $(s,t)$-switchable set generalizing ``$t$-switchable sets'',
and in this way we extend \cite[Theorem 4.13]{ST} to our framework.
However, we need several combinatorial arguments, scattered in proofs and lemmas throughout the paper,
to adapt their methods taking into account the identification among the elements of an $s$-Hankel hypermatrix.
Besides,
we also undertake new investigations, mainly 
the generalization of Watanabe's theorem, the study of the Hilbert function,
and
the combinatorics and geometry of these ideals for $s=n,n-1$.

We note that, according to \cite{Ber}, a hypermatrix $M$ is supersymmetric if it is $n$-Hankel  and 
the sizes of $M$ are  all the same, thus 
 $\It{n}{n}$ belongs to the class of ideals of varieties studied in \cite{Ber} in this special case.
 However, for arbitrary $r_i$ and $s$ we cannot speak of supersymmetric hypermatrices any longer and our ideals are more general.

All the examples in this paper have been worked out by means of \texttt{Macaulay2} \cite{M2} and in particular,
since the ideals we consider are binomial,
by means of the package \texttt{Binomials} \cite{BinM2}.
\section{Setup}\label{secSetup}

Let $\NN$ denote the set of positive integers, and if $r \in \NN$ set $[r]=\{1, \ldots, r\}$.
Given $n,r_1, \ldots, r_n\in\NN$ with $r_i \geq 2$ for each $i \in [n]$, 
define the {\bf set of indices} as
$
\mathcal{N} =~ [r_1] \times~\cdots \times ~ [r_n]
$.
For a fixed  integer $s \in [n]$ we say that 
two indices $a,b \in \mathcal{N}$ 
are {\bf $ \bf s$-equivalent} if 
$ \sum_{i=1}^s a_i = \sum_{i=1}^s b_i$  and $a_i=b_i$ for each $i=s+1, \ldots, n$.
This is indeed an equivalence relation,
and the $s$-equivalence class of an index $a$ is uniquely determined by the sum $\sum_{i=1}^sa_i$ and by the $(n-s)$-tuple $(a_{s+1}, \ldots, a_n)$.
It is  convenient to fix a  representative for each $s$-equivalence class, therefore we give the following definition:
the {\bf normal form} of an index $a \in \mathcal{N}$ 
is  
$
\overline{a}= \max\{a' \in \N : \, a' \text{ is $s$-equivalent to } a\},
$
where the maximum is  with respect to the lexicographic order on $\NN^n$.
An index $a$ is said to be in normal form if $a= \overline a$. 
So for instance if $a=(2,2,1,2)\in \mathcal{N}= [3]\times[3]\times[2]\times[2]$ and $s=2$ then $\overline{a}=(3,1,1,2)$, whereas $(3,2,1,2)$ is in normal form.
The sum of the first $s$ components of an index plays a role in proofs
and therefore we define,
more generally for a vector $a \in \mathbb{Z}^n$, the quantity $\|a\sno=  \sum_{i=1}^s |a_i|.
$

Now we fix the algebraic framework.
Let $\Bbbk$ be an arbitrary field and let $\{x_a:\, a \in \mathcal{N}\}$ be a set of variables indexed in $\mathcal{N}$ with the following identification rule: 
for any $a,b \in \mathcal{N}$ we set $x_a = x_b$ if and only if $a$ and $b$ are $s$-equivalent.
In other words, these variables are in a one-to-one correspondence with  indices in normal form.
We let 
$R= \Bbbk[x_a: \, a \in \mathcal{N}]$
be the   polynomial ring over these variables.
We  order  the variables of $R$ setting 
$
x_a < x_b$ if and only if $ \overline{a}< \overline{b},
$
where we compare $\overline a$ and $\overline b$ with respect to the lexicographic order on $\NN^n$.
We fix the lexicographic monomial order as the monomial order on $R$.

\begin{rem}\label{remDimensionR}
The Krull dimension of the ring $R$ is given by the number of indices in normal form, which equals the number of possible values of $\|a\sno$, times the number of possible values of the $(n-s)$-tuple $(a_{s+1},\ldots,a_n)$.
We obtain
$$ \dim R = \Big(\sum_{i=1}^s r_i-s+1\Big)r_{s+1}\cdots r_{n}.$$
\end{rem}

We are now ready to define the fundamental concept.

\begin{defin}\label{definSHankel}
Let $n,s,r_1, \ldots, r_n \in \NN$ and $R$ be the polynomial ring as above.
The {\bf $\bf s$-Hankel hypermatrix} is the  $r_1 \times \cdots \times r_n$ hypermatrix indexed in $\mathcal{N}$
whose $a$-entry is the variable $x_a$ in $R$,
i.e, the hypermatrix
$
M=(x_a : \, a \in \mathcal{N}).
$
\end{defin}

Now we want to introduce some determinantal ideals associated to $M$.
Unlike the case of matrices, 
different kinds of minors occur in a hypermatrix and we need to introduce more notation.
Let $L \subseteq [n]$, $a, b \in \mathcal{N}$ and define the {\bf switch of $\bf a$ and $\bf b$ with respect to $\bf L$} as the index, 
denoted by  $\sw(L,a,b)$, 
whose $i$th component is
$$
\sw(L,a,b)_i = \begin{cases}
b_i, & \text{if } i \in L; \cr
a_i, & \text{if } i \notin L.\cr
\end{cases}
$$
When $L = \{j\}$
we  just write
$ \sw(j,a,b)$.
The  {\bf Hamming distance} or simply {\bf  distance} of two indices $a,b\in \mathcal{N}$ is defined as $\di(a,b)=\#\{i\in [n] \text{ s.t. } a_i\ne b_i\}$.
Note that $\di(a,b) =
\di(\sw(L,a,b), \sw(L,b,a))$.
Given $a,b \in \N$, $i \in [n]$ and $ L \subseteq [n]$, define
the polynomials
$$
f_{L,a,b}= x_a x_b - x_{\sw(L,a,b)} x_{\sw(L,b,a)}, \qquad
f_{i,a,b}= x_a x_b - x_{\sw(i,a,b)} x_{\sw(i,b,a)}.
$$
It is easy to see that if $\di(a,b)\leq 1$ then $f_{L,a,b}=0$.
A {\bf slice minor} is an element of the form $f_{i,a,b}$ for some $i\in [n]$ and  indices $a,b \in \N$ satisfying $\di(a,b)=2$ and $a_i\ne b_i$.
A {\bf generalized minor} is an element of the form $f_{i,a,b}$ for some $i\in [n]$ and  indices $a,b \in  \N$ with arbitrary distance.
The reason for the choice of names is simple.
A nonzero  slice minor $f_{i,a,b}$ is associated to two indices $a,b$ which differ exactly in two distinct components $i,j \in [n]$:
now  $f_{i,a,b}$ is a $2\times 2$ minor of the matrix obtained from $M$ fixing all the components except  $i,j$, and
such subarray of $M$ is commonly referred to as \emph{slice} of the hypermatrix $M$.
If $\di(a,b)\geq 3$ then the element $f_{i,a,b}$ is not, in general, a minor of a slice of $M$.

We are  ready to define the ideals which are the focus of this study.

\begin{defin}\label{definIdeals}
Let $t 	\in [n]$ and  set
\begin{eqnarray*}
\I{s}{t}& =& \big(f_{i,a,b}: a, b \in \N, \di(a,b) = 2, i \in [t] \big),\\
\It{s}{t} &=&\big(f_{i,a,b}: a, b \in \N, i \in [t] \big).
\end{eqnarray*}
\end{defin}

Thus
$\I{s}{t}$ is the ideal generated by all the slice minors of $M$ such that one of the two non-fixed components is at most $t$,
whereas  $\It{s}{t}$ is generated by all the generalized minors  whose switched  component is at most $t$.

\begin{rem}\label{remtensors}
The ideal $\It{s}{t}$ admits an interpretation in terms of decomposable (i.e. rank one) tensors.
Let $V_i$  be $\Bbbk$-vector spaces of dimension $r_i$ with fixed bases $\mathcal{E}_i$ and consider the tensor product 
$V = V_1 \otimes \cdots \otimes V_n$ with the corresponding basis $\mathcal{E}$.
Grone noted in \cite{Gr} that $\It{1}{n}$ defines a variety in $\mathbb{P}(V)$ which parametrizes all decomposable tensors in $V$. 
More generally, in \cite{ST} $\It{1}{t}$ is viewed as the ideal cutting out all decomposable tensors in flattenings of $V$ of the form $V_i \otimes ( \otimes_{j\ne i} V_j)$ as $i$ varies in $[t]$.

We may say that a tensor $v\in V$ is $s$-Hankel if its components with respect to $\mathcal{E}$ satisfy the same relations as the variables $x_a$;
these tensors determine a linear subspace $H \subseteq V$.
Thus the ideal $\It{s}{n}$ defines a variety in $\mathbb{P}(H) \subseteq \mathbb{P}(V)$ which parametrizes decomposable $s$-Hankel tensors in $V$.
Similarly, the ideal $\It{s}{t}$ defines a variety in $\mathbb{P}(H)$ parametrizing all decomposable $s$-Hankel tensors in flattenings of $V$ of the form $V_i \otimes ( \otimes_{j\ne i} V_j)$ as $i$ varies in $[t]$.

A similar description is provided in \cite{Ber} for a class of symmetric tensors parametrized by so-called Segre-Veronese varieties.
\end{rem}

 It is important to observe that the ideals associated to the $s$-Hankel hypermatrix 
 can also be interpreted as determinantal ideals in the classical sense of matrices. 
 This is obvious for  $\I{s}{t}$, 
 which is generated by the minors of certain slices. 
 For  $\It{s}{t}$ we explain the relationship with matrices in the following discussion.
 This discussion also  plays a role in Section 6. 
 
\begin{discuss}\label{discFlattening}
A generalized minor $f_{i,a,b}$ with any value of $\di(a,b)$ can be seen as a minor of a suitable matrix, constructed as follows.
We rearrange the entries $x_a$ of the hypermatrix $M$ in a matrix $M'$ whose rows are indexed by the $i$-component of $a$  and columns by the remaining $n-1$ components:
the matrix $M'$ is a {\bf flattening} of the hypermatrix $M$ with respect to the component $i$.
Thus $M'$ has $r_i$ rows and $r_1\cdots \widehat{r_{i}}\cdots r_n$ columns
and now   $f_{i,a,b}$  is  the $2\times 2$ minor of $M'$ determined by the indices $a,b$.
We obtain that $\It{s}{t}$ is the sum of the $t$ determinantal ideals of the flattenings of $M$ with respect to each of the first $t$ components.
\end{discuss}

\begin{rem}\label{remTminusOne}
We note that values $t=n$ and $t=n-1$ define the same ideals, i.e.
$\I{s}{n}=\I{s}{n-1}$ and $ \It{s}{n}=\It{s}{n-1}$.
This is easy for the ideal $\I{s}{t}$: 
if $a,b \in \mathcal{N}$ with $\di(a,b)=2$ and $a_i \ne b_i, a_j \ne b_j$, then for example we have $j \leq n-1$ and $f_{i,a,b}=f_{j,a,b}\in \I{s}{n-1}$ proving that 
$\I{s}{n}=\I{s}{n-1}$.
It is easy to see that the following equations hold:
$$
f_{n,a,b}
=
f_{[n-1],a,b}= \sum_{j=1}^{n-1}
f_{j, \sw(\{1, \ldots, {j-1}\},a,b), \sw(\{1, \ldots, {j-1}\},b,a)}
$$
and each summand in the right belongs to $\It{s}{n-1}$. Thus  $\It{s}{n}=\It{s}{n-1}$.
\end{rem}

We conclude this section with a  result that will be useful later.

\begin{lem}\label{lemProductMonomial}
Let $i\in [t]$ and
  $a=c_0, c_1, \ldots, c_k, b \in \mathcal{N}$.
Suppose that for all $j \in [k]$ we have $\di(c_{j-1},c_j)=1$,
with $c_{j-1},c_j$ differing in position $l_j \ne i$.
Suppose also that  $\di(c_k,b) = 2$ and $c_{k,i}\ne b_i$.
Then
$
x_{c_1} x_{c_2} \cdots x_{c_k} f_{i,a,b}
\in \I{s}{t}.
$
\end{lem}
\begin{proof}
The  proof given in  \cite[Lemma 4.10]{ST} for the case  $s=1$ works for any value of $s$.
\end{proof}

\section{$(s,t)$-Switchable sets}\label{secSwitch}

\begin{defin}\label{definSTgood}
Let $s,t \in [n]$.
A subset $\Sw \subseteq \N$ is  {\bf $\bf(s,t)$-switchable}  if the following two properties hold:
\begin{enumerate}
\item for all $a,b \in \Sw$ with $\di(a,b) = 2$ and
$i \in [t]$ we have $\sw(i,a,b) \in \Sw$;
\item if $a \in \Sw$ and  $b \in \N$ is $s$-equivalent to $a$ then $b \in \Sw$.
\end{enumerate}
\end{defin}

It is straightforward that $\emptyset$ and $\N$ are $(s,t)$-switchable sets for any values of $s,t$.
Moreover for each $t\in [n-1]$ an $(s,t+1)$-switchable set is also an $(s,t)$-switchable set and
the converse holds if $t=n-1$, but not in general.
Note that property (1) from Definition \ref{definSTgood} is equivalent to the following condition:
for any $a,b \in \N$
and any distinct $i \in [t]$ and $j \in [n]$ we have
$a, \sw(\{i,j\},a,b) \in \Sw$
if and only if $\sw(i,a,b), \sw(j,a,b) \in \Sw$.

Let $\Sw \subseteq \N$.
Two indices $a, b \in \Sw$ are {\bf connected} in $\Sw$
if there exist $c_0 = a, c_1, \ldots, c_k = b \in \Sw$
such that for all $j \in [k]$ we have
$\di(c_{j-1},c_j)=1$;
the sequence $c_0, \ldots,  c_k$ is called a {\bf path} between $a$ and $b$.
Clearly connectedness is  an equivalence relation on $\Sw$.
The next lemma shows that one of the two equivalence relations defined on $(s,t)$-switchable sets is coarser than the other.

\begin{lem}\label{lemSequivalenceImpliesConnectedness}
Let $\Sw$ be an $(s,t)$-switchable set. If $a,b\in \Sw$ are $s$-equivalent, then they are connected in $\Sw$.
In other words, $s$-equivalence implies connectedness.
\end{lem}
\begin{proof}
We prove the lemma by induction on the quantity $\delta(a,b)=\|a-b\sno = \sum_{i=1}^s|a_i-b_i|$.
Notice that since $a,b$ are $s$-equivalent then $\sum_{i=1}^sa_i = \sum_{i=1}^sb_i$ and thus $\delta(a,b)$ is a non-negative even integer, and $a$ and $b$ differ either in no components or in at least two.
If $\delta(a,b)=0$ then $a=b$ and they are trivially connected.

Suppose $\delta(a,b)=2$.
Then there are exactly two components $h,k\in [s]$ in which $a,b$ differ, and we must have $b_h=a_h+1$ and $b_k=a_k-1$.
If $1 \in \{ h,k\}$, that is to say, if $a_1 \ne b_1$, 
then  $c=\sw(1,a,b)\in \Sw$ and $a,c,b$ is a path connecting $a$ and $b$ since $\di(a,c)=\di(b,c)=1$.
Assume now that $a_1 = b_1$. 
The idea is to ``move'', via $s$-equivalence,
 the difference between the $h$th or the $k$th components of $a$ and $b$ to the first component, so that we can switch it no matter what $t$ is.
We  distinguish two cases.
\begin{itemize}
\item Case $a_1=1$. Then $a$ is $s$-equivalent to $a'=(2,a_2, \ldots, a_k-1,\ldots, a_n)$. 
Since $\Sw$ is $(s,t)$-switchable we have $a'\in S$ and also $c=\sw(1,a',a)\in \Sw$.
It is immediate to check that $\di(a,c)=\di(b,c)=1$ so that $a,c,b$ is a path between $a$ and $b$.

\item Case $a_1>1$. Then $a$ is $s$-equivalent to $a'=(a_1-1,a_2, \ldots, a_h+1,\ldots, a_n)$. 
Similarly to the previous case, $a'$ and $c=\sw(1,a',a)$ belong to $\Sw$ and $a,c,b$ is a path.
\end{itemize}

Finally, assume $\delta(a,b)>2$.
Then there exist  $h,k\in [s]$ such that $a_h>b_h$ and $a_k<b_k$. 
We have that $a$ is $s$-equivalent to $a'=(a_1, \ldots, a_h-1, 	\ldots, a_k+1, \ldots, a_n)$, and thus $a' \in \Sw$.
It is immediate to check that $\delta(a,a')=2$ and $\delta(b,a')<\delta(b,a)$ and the proof is completed by induction.
\end{proof}

\begin{cor}\label{corEquivalenceClassItselfGood}
  Let $\mathcal{T}$ be a connected component in an $(s,t)$-switchable set $\Sw$.
Then $\mathcal{T}$ is itself an $(s,t)$-switchable set.
\end{cor}
\begin{proof}
Let $a,b \in \mathcal{T}$ with $\di(a,b)=2$ and  $i \in [t]$.
 Then $\di(a, \sw(i,a,b)) \leq 1 $ and $\sw(i,a,b)$ is connected to $a$,
so that $\sw(i,a,b)\in \mathcal{T}$.
By Lemma \ref{lemSequivalenceImpliesConnectedness} $\Tw$ is closed under $s$-equivalence.
\end{proof}

Let $\Sw\subseteq \N$ be an $(s,t)$-switchable set.
We define  three ideals of $R$:
\begin{eqnarray*}
\It{s}{t}_\Sw & = & (f_{i,a,b} : i\in [t], a,b\ \text{connected in } \Sw),\\
 \Var{s}{t}_\Sw & = & (x_a : a \not \in \Sw),\\
 \P{s}{t}_\Sw & = & \Var{s}{t}_\Sw + \It{s}{t}_\Sw.
\end{eqnarray*}
Notice that both ideals $\It{s}{t}_\Sw$ and $\P{s}{t}_\Sw$ generalize the ideal 
$\It{s}{t}$ as  $\It{s}{t}=\It{s}{t}_\N=\P{s}{t}_\N$.

\begin{lem}\label{lemPcontainsI}
Let $\Sw$ be an $(s,t)$-switchable subset of $\N$.
The ideal $\P{s}{t}_\Sw$ contains $\I{s}{t}$.
\end{lem}

\begin{proof}
The proof of \cite[Proposition 4.2]{ST} for the case $s=1$ works for any value of $s$.
\end{proof}

The following technical lemma will be needed to prove the main result of Section \ref{secI}.

\begin{lem}\label{lemPathAB}
Let $\Sw$ be an $(s,t)$-switchable set. 
Let $a,b \in \Sw$ be  connected and $i\in [t]$ a component such that $a_i\ne b_i$.
Then there exist elements $a=c_0,c_1, \ldots, c_l\in \Sw$ such that 
$\di(c_l,b) = 2$,
$c_{l,i}\ne b_i$,
$\di(c_{j-1},c_j)=1$ 
and 
$c_{j-1},c_j$ do not differ in component $i$
for all $j \in [l]$.
\end{lem}
\begin{proof}
By hypothesis there is a path $a=e_0, e_1, \ldots, e_l,e_{l+1}, e_{l+2}=b$ in $\Sw$ connecting $a,b$.
Leaving out the last two indices in the path, 
there are indices $a=e_0, e_1, \ldots, e_l\in \Sw$ such that $\di(e_{j-1},e_j)=1$ for all $j\in [l]$ and $\di(e_l,b)=2$.

Let us consider the indices $c_j=\sw(i,e_j,a)$ for all $j=0,\ldots, l$.
We claim that these indices belong to $\Sw$;
this is trivially true for $j=0$.
If $j>0$ then $c_j=\sw(i,e_j,e_0)=\sw(i,e_j,c_{j-1})$ where $\di(e_j,c_{j-1})=\di(e_j,\sw(i,e_{j-1},e_0))\leq 2$; 
 by induction on $j$ we have
  $c_j \in \Sw$.

Now the $i$-component of $c_j$ is equal to $a_i$ for all $j\in [l]$, and in particular $c_{l,i} \ne b_i$.
We also have $\di(c_{j-1},c_j)\leq 1$ for all $j\in [l]$ and up to pruning redundant elements we may assume $\di(c_{j-1},c_j)= 1$.
All we have to check is the distance $\di(c_l,b)$: if this distance is equal to $2$ then the elements $c_j$ satisfy all the desired properties.
If $\di(c_l,b)\ne 2$,
since $\di(c_l,b)=\di(\sw(i,e_l,a),b)$ and $\di(e_l,b)=2$ then we must have either
$\di(c_l,b)=1$ or $3$ (we are changing only one component in $e_l$).
The case 
$\di(c_l,b)=1$
implies $c_{l,i}=b_i$ which is a contradiction.
If 
$\di(c_l,b)=3$,
then $e_{l,i}=b_i$; in this case we add to the sequence the element $c_{l+1}=\sw(i, e_{l+1},a)$.
Since $e_{l,i}=b_i$, $\di(e_l,b)=2$ and $\di(e_{l},e_{l+1})=\di(e_{l+1},b)=1$ then necessarily $e_{l+1,i}=b_i$.
Since  $\di(e_{l+1},b)=1$ and $e_{l+1,i}=b_i\ne a_i$ then $\di(c_{l+1},b)=2$ and they differ in the $i$-component.
Of course we have $\di(c_l, c_{l+1})=0$ or $1$, but  $\di(c_{l+1},b)\ne \di(c_{l},b)$ implies  
$\di(c_l, c_{l+1})=1$.
The proof is completed.
\end{proof}

\section{Structure of the ideal $\It{s}{t}$}\label{secITilde}

Let $\Sw$ be a connected  $(s,t)$-switchable set.
The main aim of this section is to characterize in Theorem \ref{thmReductGeneral} when two monomials are equivalent modulo ${\It{s}{t}}_\Sw$.
We point out that the proof of this result, which is of fundamental importance in the paper,
is not a simple rewriting of the proof of \cite[Lemma 6.2]{ST} for the generic case.
 We prove consequently that the  set $G_\Sw= \big\{ f_{K,a,b} : \, K \subseteq [t], a,b \in \Sw \big\}$ is a Gr\"obner basis for $\It{s}{t}_\Sw$ and we derive  information on $\It{s}{t}_\Sw$ such as primeness and the Hilbert function.

In this section we use the term {\bf multiset}  to indicate a finite list where elements are counted with multiplicity and  order is irrelevant. 
When we say ``reduction'' we mean in the sense of the Gr\"obner bases.
The symbol $\lt(f)$ denotes the leading term of a polynomial $f \in R$.
Recall that we  fixed the lexicographic order on the monomials of $R$.

 We observe that an element $g \in G_\Sw$ has the form $g=x_{a_1}x_{a_2}-x_{b_1}x_{b_2}$, for 
 some  indices $a_i\in \Sw$ and indices  $b_1=\sw(K,a_1,a_2)$, $b_2=\sw(K,a_2,a_1)$, with $K \subseteq [t]$. 
By definition of switch, 
it is easy to check that these indices $a_i, b_i$ 
satisfy the following properties:
\begin{itemize}
\item $\|a_1\sno + \|a_2\sno = \|b_1\sno + \|b_2\sno $;

\item if $ s \leq t$: the multiset $\{a_{1,i}, a_{2,i}\}$ is the same as the multiset $\{b_{1,i}, b_{2,i}\}$ for each  $i=s+1, \ldots, t,$
 and the multiset
 $\{ (a_{1,t+1},  \ldots, a_{1,n}),
(a_{2,t+1},  \ldots, a_{2,n})\}$
is the same as the multiset 
$\{(b_{1,t+1},  \ldots, b_{1,n}),
(b_{2,t+1},  \ldots, b_{2,n})\}$;

\item if $ s \geq t$: the multiset
 $\{ (a_{1,s+1},  \ldots, a_{1,n}),
(a_{1,s+1},  \ldots, a_{2,n})\}$
is the same as the multiset 
$\{(b_{1,s+1},  \ldots, b_{1,n}),
(b_{2,s+1},  \ldots, b_{2,n})\}$.
\end{itemize}
Note that the last two conditions  are the same when $s=t$.
Generalising a bit,  take now an element of the form $h=\alpha g$, 
where $g\in G_\Sw$ and $\alpha$ is a monomial of $R$.
We observe that $h$ has the form 
 $
\prod_{i=1}^d x_{a_i}- \prod_{i=1}^d x_{b_i}
 $
with indices $a_i,b_i \in \N$
that satisfy the following properties:

 \begin{itemize}
\item[(1)]  $\sum_{i=1}^d \|a_i\sno =\sum_{i=1}^d \|b_i\sno$;

\item[(2a)] if $ s \leq t$:  the multiset $\{a_{1,i},\ldots, a_{d,i}\}$ is the same as the multiset $\{b_{1,i},\ldots, b_{d,i}\}$ for   $i=s+1, \ldots, t,$ and the multiset
 $\{ (a_{1,t+1},  \ldots, a_{1,n}),
 \ldots,$
 $(a_{d,t+1},  \ldots, a_{d,n})\}$
is the same as the multiset 
$\{(b_{1,t+1},  \ldots, b_{1,n}),\ldots,
(b_{d,t+1},  \ldots, b_{d,n})\}$;

\item[(2b)] if $s \geq t$: the multiset
 $\{ (a_{1,s+1},  \ldots, a_{1,n}),
 \ldots,
(a_{d,s+1},  \ldots, a_{d,n})\}$
is the same as 
$\{(b_{1,s+1},  \ldots, b_{1,n}),\ldots,
(b_{d,s+1},  \ldots, b_{d,n})\}$.
\end{itemize}
It follows that every reduction step of a monomial  $\alpha = \prod_{i=1}^d x_{a_i} $ with respect to $G_\Sw$ preserves the quantity of property (1) and the multisets of properties (2a)-(2b) from the list above.
In particular, 
if  a homogeneous binomial $p= \prod_{i=1}^d x_{a_i}- \prod_{i=1}^d x_{b_i}$  reduces to 0 modulo~$G_\Sw$
 then necessarily these quantities and multisets are the same for the first and the second monomial.
We summarize this discussion in the following proposition,
where we focus on binomials $p$ involving only variables $x_a$ with $a\in \Sw$.

\begin{prop}\label{propNecessaryReduct}
Let $\Sw$ be a connected  $(s,t)$-switchable set.
Let $p=\prod_{i=1}^d x_{a_i}- \prod_{i=1}^d x_{b_i}$ with  $a_i, b_i\in \Sw$.
Assume that $p$ reduces to $0$ with respect to the set $G_\Sw$. 
Then the indices $a_i, b_i$ satisfy the property (1) and one of (2a) and (2b) on the list above.
\end{prop}

With this notation, our aim is to prove the converse of Proposition \ref{propNecessaryReduct} when $s\geq 2$.

\begin{lem}\label{lemOrderVars}
Let $a,b \in \N.$
If $ \|a\sno > \|b\sno$, then $x_a > x_b$.
\end{lem}
\begin{proof}
By definition of the order on variables, we have to compare $\overline{a}$ and $\overline{b}$.
Notice that we have $\|\overline{a}\sno=\|a\sno$, $\|\overline{b}\sno=\|b\sno$ and hence $ \|\overline{a} \sno > \| \overline{b}\sno$.
But this implies that the first non-zero component in $\overline a- \overline b$ is positive, hence $\overline a> \overline b$ with respect to the lexicographic order on $\N$, and thus $x_{\overline a}=x_a > x_b =x_{ \overline b}$.
\end{proof}

\begin{lem}\label{lemReduc1}
Let $\Sw$ be a connected  $(s,t)$-switchable set, $a,b \in \Sw$  and  $s \geq 2$. 
If $\big|\|a\sno -\|b\sno \big|\geq 2$ then $x_a x_b$ is not reduced modulo $G_\Sw$.
\end{lem}
\begin{proof}
Assume without loss of generality that $\|a\sno \geq \|b\sno$.
Then the hypothesis becomes $\|a\sno -\|b\sno \geq 2$,
or in terms of components $\sum_{i=1}^s(a_i-b_i) \geq 2$.
This guarantees the existence of $a',b'\in \N$, that are $s$-equivalent to $a$ and $b$, respectively, such that
\begin{equation}\label{ineq1}\tag{$\star$}
a'_1 -b'_1\geq 1, \qquad \sum_{i=2}^s (a'_i-b'_i)\geq 1.
\end{equation}
Since $a'$ is $s$-equivalent to $a$ we have $a'\in \Sw$ and
similarly $b' \in \Sw$, so that $g=f_{1,a',b'} \in G_\Sw$.
By the inequalities (\ref{ineq1}) we easily have  $\|a'\sno > \|\sw(1,a',b')\sno, \|\sw(1,b',a')\sno$ and  by Lemma \ref{lemOrderVars} it follows $\lt(g)=x_{a'}x_{b'}= x_a x_b$. Thus
$x_a x_b$ 
can be reduced with respect to $g \in G_\Sw$.
\end{proof}

\begin{lem}\label{lemReducHankel}
Let $\Sw$ be a connected $(s,t)$-switchable set and assume $s \geq 2$.  
Consider a binomial
$p = \prod_{i=1}^d x_{a_i}- \prod_{i=1}^d x_{b_i}$ with $a_i,b_i\in \Sw$ such that
\begin{equation}\label{eq1}\tag{$\star \star$}
\sum_{i=1}^d \|a_i\sno= \sum_{i=1}^d \|b_i\sno.
\end{equation}
If  $p$ is reduced modulo $G_\Sw$ then, up to reindexing, it satisfies
the following conditions
\begin{enumerate}
\item
 $\|a_i\sno=\|b_i\sno$ for each $i \in [d]$;

\item there exists $k\in [d]$ such that $\|a_1\sno=\cdots =\|a_k\sno>\|a_{k+1}\sno=\cdots =\|a_d\sno$ and $\|a_1\sno=\|a_{k+1}\sno+1$
(if $k=d$ then all the $\|a_i\sno$ are equal).
\end{enumerate}
\end{lem}

\begin{proof}
Replacing the indices $a_i, b_i$ with the normal forms $\overline{a}_i, \overline{b_i}$ does not affect equality (\ref{eq1}) nor changes the variables involved,
therefore we may assume that all the indices are in normal form.
This implies the simple formula $\big| \| a \sno  - \|b\sno \big| = \|a-b\sno$.
Assume that $p$ is  reduced.

If it happens that $\|a_i- a_j\sno\geq 2$ for some $i,j\in [d]$, 
then by Lemma \ref{lemReducHankel} 
 $x_{a_1} x_{a_2} \cdots x_{a_d}$ is not reduced,  contradiction.
Therefore for  all $i,\,j\in [d]$ we have  $\|a_i-a_j\sno=0$ or $1$,
and the same fact holds for the $b_i$.
In particular there may be at most $2$ distinct values of $\|a_i\sno$ 
(otherwise there would be $a_i, a_j$ with $\|a_i-a_j\sno \geq 2$).
We state this more precisely:
up to reordering the $a_i, b_i$,
there are integers $k,k' \in [d]$ such that
\begin{eqnarray*}
\|a_1\sno=\cdots = \|a_k\sno = \|a_{k+1}\sno +1 = \cdots = \|a_d\sno+1, \\
\|b_1\sno=\cdots = \|b_{k'}\sno = \|b_{k'+1}\sno +1= \cdots = \|b_d\sno +1.
\end{eqnarray*}
Of course the condition $\|a_{1}\sno=\|a_{k+1}\sno+1$ plays a role only if $k<d$, and similarly for the $b_i$.
Call for brevity $A=\|a_1\sno$ and $B=\|b_1\sno$.
If we substitute these equalities in equation (\ref{eq1})
we obtain 
$
kA +(d-k)(A-1) = k'B+(d'-k')(B-1), $
and after easy manipulations we get
$ d(A-B)= k' -  k$.
Since $k,k'\in [d]$ we get $|k'-k|\leq d-1$ and the only possibility is $A-B=k'-k=0$,
which implies the conclusion.
\end{proof}

\begin{thm}\label{thmReductGeneral}
Let $\Sw$ be a connected  $(s,t)$-switchable set and assume $s \geq 2$.
A  binomial $p = \prod_{i=1}^d x_{a_i}- \prod_{i=1}^d x_{b_i}$, with $a_i, b_i \in \Sw$,
reduces to $0$ with respect to $G_\Sw$ if and only if the following properties hold:
\begin{itemize}
\item[(1)]  $\sum_{i=1}^d \|a_i\sno =\sum_{i=1}^d \|b_i\sno$;

\item[(2a)] if $ s \leq t$:  the multiset $\{a_{1,i},\ldots, a_{d,i}\}$ is the same as the multiset $\{b_{1,i},\ldots, b_{d,i}\}$ for   $i=s+1, \ldots, t,$ and the multiset
 $\{ (a_{1,t+1},  \ldots, a_{1,n}),
 \ldots,
(a_{d,t+1},  \ldots, a_{d,n})\}$
is the same as the multiset 
$\{(b_{1,t+1},  \ldots, b_{1,n}),\ldots,
(b_{d,t+1},  \ldots, b_{d,n})\}$;

\item[(2b)] if $s\geq t$: the multiset
 $\{ (a_{1, s+1},  \ldots, a_{1,n}),
 \ldots,
(a_{d,s+1},  \ldots, a_{d,n})\}$
is the same as  
$\{(b_{1,s+1},  \ldots, b_{1,n}),\ldots,
(b_{d,s+1},  \ldots, b_{d,n})\}$.
\end{itemize}
\end{thm}

\begin{proof}
The necessary part was proved in Proposition \ref{propNecessaryReduct}.
We may assume that $p$ is already reduced with respect to $G_\Sw$. 
 We proceed by induction on the degree $d$ of the binomial.
If $d = 1$ then
necessarily $p = 0$;
now suppose that $d> 1$.
The cases $s \leq t $ and $s\geq t$ are somehow different, and 
we treat them separately (the two arguments coincide if $s=t$).

Let us deal with the case $s \leq t$.
In the proof of this case we enumerate the set $[r_t+1 ] \times \cdots \times  [r_n ]$ preserving the (lexicographic) order on $(n-t)$-tuples,
so that we can treat the last $ n - t $ components of elements of $\N$ as one component.
In other words, we may assume without loss of generality that $ t = n - 1 $.
Note that,
since an $(s,n)$-switchable set is the same thing as an $(s,n-1)$-switchable set, this also covers
the case $t = n$.
When we write $f_{L,a,b}$ and $L$ is a subset of $[n]$ containing $n=t+1$  we actually mean 
$f_{[t]\setminus L,a,b}$ in the usual sense.
If all the elements in each of the multisets of property (2a) are equal, 
then clearly the two monomials are equal too and $p=0$.
If there are some different elements in a multiset, 
we claim that the  minimal elements in each multiset are those belonging to indices $a_1,\ldots,a_l$ (where $l<d$ is the number of minimal element in the fixed multiset).
In fact,
if we had $a_{j,i} > a_{h,i}$ for some $j\leq l$, $h>l$, and $i \in [n]$, then we could reduce $p$ with respect to $f_{i,a_j,a_k}\in G_\Sw$, 
contradicting the fact that $p$ is reduced. 
But now $a_1, \ldots, a_l$ and $b_1, \ldots, b_l$ satisfy the hypothesis of the theorem, and so do $a_{l+1}, \ldots, a_d$ and $b_{l+1}, \ldots, b_d$ so that the conclusion follows by induction on $d$.

Let us deal with the case $s \geq t$.
We apply Lemma \ref{lemReducHankel} and assume that conditions (1) and (2) from
 that statement  are satisfied.
If $k=d$ in the notation of Lemma \ref{lemReducHankel} then it follows immediately that the two monomials are equal and thus $p=0$.
Assume $k<d$.
If all the elements in the multiset of property (2b)
 are equal, then they are also equal for the  $b_i$, and we have $p=0$.
If they are not all the same, we claim that the $k$ minimal elements   must appear as the last $n-s$ components of the tuples $a_1, \ldots, a_k$, and the same for the $b_i$.
Assume by contradiction that this is not the case.
Then we  have for example for the first monomial
$$(a_{j,s+1}, \ldots, a_{j,n})>(a_{h,s+1},  \ldots, a_{h,n})$$
for some $j \leq k < h$.
But then we may reduce this monomial with respect to the element 
$g=f_{1,a'_j,a'_h} \in G_\Sw$ for suitable choices of $a'_j,a'_h$.
Here the idea is to take $s$-equivalent indices where we ``moved'' the only difference in the first $s$ components to the first component.
Precisely, we choose $a'_j,a'_h\in \N$ such that
\begin{itemize}
\item $a'_j$ and $a'_h$ are $s$-equivalent to $a_j, a_h$ respectively;

\item $a'_{j,i}=a'_{h,i}$ for $i=2, \ldots, s$;

\item $a'_{j,1}=a'_{h,1}+1$.
\end{itemize}
These choices of indices are possible because $\|a_j\sno=\|a_h\sno+1$.
For these choices we easily have
\begin{itemize}
\item $\|a'_j\sno> \|\sw(1,a'_j,a'_k)\sno$ so that $x_{a'_j}> x_{\sw(1,a'_j,a'_k)}$ by Lemma 	\ref{lemOrderVars};

\item $\|a'_j\sno= \|\sw(1,a'_k,a'_j)\sno, (a_{j,s+1}, \ldots, a_{j,n})>(a_{h,s+1},  \ldots, a_{h,n})$ so that  $x_{a'_j}> x_{\sw(1,a'_k,a'_j)}$.
\end{itemize}
It follows that $\lt(g)=x_{a_j}x_{a_k}$ and we can reduce $p$,
 reaching a contradiction to $p$ being reduced.
Thus the $k$ minimal elements must appear as the last $n-s$ components of the tuples $a_1, \ldots, a_k$, and the same for the $b_i$;
but now $a_1, \ldots, a_k$ and $b_1, \ldots, b_k$ satisfy the hypothesis of the theorem, and so do $a_{k+1}, \ldots, a_d$ and $b_{k+1}, \ldots, b_d$ so that the conclusion follows by induction on $d$.
\end{proof}

\begin{cor}\label{thmGroebner}
Let $\Sw$ be an $(s,t)$-switchable set and assume $s\geq 2$.
Then the following set 
is a Gr\"obner basis for $\It{s}{t}_\Sw$:
$$
G_\Sw = \big\{f_{K,a,b} : K \subseteq [t], a, b \text{ connected in } \Sw \big\}.
$$
\end{cor}

\begin{proof}
We partition $\Sw$ into its connected components
$\Sw = \Sw_1 \cup \cdots \cup \Sw_r $
and by Corollary \ref{corEquivalenceClassItselfGood} each $\Sw_i$ is itself an $(s,t)$-switchable set.
The set $G_\Sw$ may be expressed as the union $ G_{\Sw_1} \cup \cdots \cup G_{\Sw_r}$,
where
$G_{\Sw_i} = \{f_{K,a,b} : K \subseteq [t], a, b \in \Sw_i\}.$
It is clear that each generator of $ \It{s}{t}_{\Sw_i}$  belongs to $G_{\Sw_i}$, so that we have the inclusion $ \It{s}{t}_{\Sw_i} \subseteq (G_{\Sw_i})$.
Conversely, let  $a, b \in \Sw_i$ and $K = \{k_1, \ldots, k_l\} \subseteq [t]$.
Expanding the sum we may check that the element $f_{K,a,b}$ can be obtained as 
$$
f_{K,a,b}= \sum_{j=1}^l
f_{k_j, \sw(\{k_1, \ldots, k_{j-1}\},a,b), \sw(\{k_1, \ldots, k_{j-1}\},b,a)}.
$$
It is easy to see by induction on $j$ that $\sw(\{k_1, \ldots, k_{j-1}\},a,b),$ $ \sw(\{k_1, \ldots, k_{j-1}\},b,a)\in \Sw_i$ and  they are all connected in $\Sw_i$ by assumption.
Each of the summands on the right hand side is thus a generator of $\It{s}{t}_{\Sw_i}$,
so that $G_{\Sw_i} \subseteq \It{s}{t}_{\Sw_i}$.
Therefore we have  $ \It{s}{t}_{\Sw_i} =(G_{\Sw_i})$, and taking the sum of these ideals we conclude that $G_\Sw$ is a system of generators for  $\It{s}{t}_\Sw$.

Now we apply Buchberger's criterion to prove that $G_\Sw$ is a Gr\"obner basis, therefore we consider S-polynomials of pairs of elements $f,g\in G_\Sw$.
If $f \in G_{\Sw_i}$, $g \in G_{\Sw_j}$ with $i \ne j$
then the S-polynomial trivially reduces to $0$
with respect to $G_\Sw$
as the two sets of variables appearing in $f$ and $g$ are disjoint.
If $f, g\in G_{\Sw_i}$ for some $i$ then the S-polynomial is either $0$ or a  binomial
satisfying Theorem \ref{thmReductGeneral} and 
hence it reduces to $0$  with respect to $G_\Sw$.
\end{proof}

Another application  is the following result, 
in which we adopt the convention that an empty product is equal to 1. 

\begin{cor}\label{corHilbert}
Let $H(d)$ denote the Hilbert function of the algebra $ R / \It{s}{t}$.
 If $2 \leq s \leq t$ then
$$
H(d)= \Big(d \sum_{j=1}^s (r_j -1)+1\Big){ d+r_{t+1}\cdots r_{n}-1\choose d} \prod_{i=s+1}^t { d+r_i-1\choose d},
$$
whereas if $s \geq t$, $s \geq 2$ then
$$
H(d)= \Big(d \sum_{j=1}^s (r_j -1)+1\Big){ d+r_{s+1}\cdots r_{n}-1\choose d}.
$$
\end{cor}

\begin{proof}
Write $R/\It{s}{t} = \oplus_{d \geq 0} (R/\It{s}{t})_d$, 
where $(R/\It{s}{t})_d$
is the graded component of degree $d$.
Since this $\Bbbk$-vector space is generated by all the monomials of degree $d$,
in order to compute its dimension 
we only need to count the number of distinct monomials of degree $d$  in $R/{\It{s}{t}}$.
If $s \geq 2$, 
by
Theorems \ref{thmReductGeneral} and Corollary \ref{thmGroebner}  
(applied to the $(s,t)$-switchable set $\N$) 
a monomial $ \prod_{i=1}^d x_{a_i}$ is uniquely determined by 
\begin{itemize}
\item if $s \leq t$: the quantity $\sum_{i=1}^d \|a_i\sno $,  the multisets $\{a_{1,i},\ldots, a_{d,i}\}$ for  $i=s+1, \ldots, t,$ and the multiset
 $\{ (a_{1,t+1},  \ldots, a_{1,n}),
 \ldots,$
$(a_{d,t+1},  \ldots, a_{d,n})\}$;

\item if $s\geq t$: 
the quantity $\sum_{i=1}^d \|a_i\sno $,
the multiset
 $\{ (a_{1, s+1},  \ldots, a_{1,n}),
 \ldots,$
$(a_{d,s+1},  \ldots, a_{d,n})\}$.
 \end{itemize}
Multisets can be counted easily via binomial coefficients: 
a multiset of size $k$ from a set of $n$ elements can be chosen in $ { n+k-1 \choose k}$ different ways.
The number of choices for $\sum_{i=1}^d \|a_i\sno $ is clearly determined by the numbers $r_1, \ldots, r_n$.
Precisely, we have:
\begin{itemize}
\item the  quantity $\sum_{i=1}^d \|a_i\sno $ can be chosen in 
$
\big(\sum_{i=1}^d \sum_{j=1}^s r_j - \sum_{i=1}^d \sum_{j=1}^s 1 \big) +1 =d \sum_{j=1}^s (r_j -1)+1 
$
ways;

\item if $i\in [t]$ the  multiset $\{a_{1,i},\ldots, a_{d,i}\}$ can be chosen in 
$
{ d+r_i-1\choose d} ={ d+r_i-1\choose r_i-1} 
$
ways;

\item the multiset of the final $n-t$ (resp. $n-s$) components can be chosen in 
$
{d+r_{t+1}\cdots r_{n}-1\choose d} = {d+r_{t+1}\cdots r_{n}-1\choose r_{t+1}\cdots r_{n}-1} 
$
ways (resp. $
{d+r_{s+1}\cdots r_{n}-1\choose d} = {d+r_{s+1}\cdots r_{n}-1\choose r_{s+1}\cdots r_{n}-1} 
$
ways).
\end{itemize}
The total number of distinct monomials of degree $d$ is simply obtained by multiplying these expressions,
and we get exactly the  expressions of the statement.
\end{proof}

\begin{rem}
The Hilbert function for $s=1$ was found in \cite[Proposition 1.9]{Ha}. 
Interestingly, 
that expression is not a particular case of the ones found here for higher values of $s$;
in other words, the generic case behaves differently from that of other $s$-Hankel hypermatrices.
\end{rem}

Finally, Theorem \ref{thmReductGeneral} gives the primeness of the ideal $\It{s}{t}$.

\begin{prop}\label{propPrimeSingle}
If $\Sw$ is a connected  $(s,t)$-switchable set,
then $\It{s}{t}_S$ is prime.
\end{prop}
\begin{proof}
Assume by contradiction that $\It{s}{t}_\Sw$ is not prime.
We suppose first that $\Bbbk$ is  algebraically closed.
By  \cite[Theorem 6.1]{ES}
there exists a polynomial $f$ and a binomial $g$ in $R$ such that 
$f,g \notin \It{s}{t}_\Sw$ and $fg\in \It{s}{t}_\Sw$.
Write $g=\alpha -  c \beta $ where $\alpha, \beta$ are monomials and $c\in \Bbbk$, and 
$f = d_1 \gamma_1 + d_2 \gamma_2 + \cdots + d_k \gamma_k$ where $d_i \in \Bbbk$ and $\gamma_i$ are monomials.
Assume $\alpha= \lt(g) $ and $\gamma_1 = \lt(f)$.
Assume that every monomial in $f$ and $g$ is reduced modulo the Gr\"obner basis $G_\Sw$ of $\It{s}{t}_\Sw$.
We are going to use the characterization from Theorem \ref{thmReductGeneral}, therefore we assume $s\geq 2$;
however $\It{1}{t}$ is prime because it is a Segre ideal (cf. \cite{ST}).

Since $ fg \in \It{s}{t}_\Sw$ then it  reduces to $0$ modulo $G_\Sw$.
It follows that its leading term 
$\alpha \gamma_1$ reduces modulo $G_\Sw$ to the same monomial as some other monomial in $fg$.
If it reduces to the same monomial as $\alpha \gamma_i$ for some $i > 1$,
then $ \alpha \gamma_1$ and $\alpha \gamma_i$ satisfy the hypotheses of Theorem  \ref{thmReductGeneral}, hence $ \gamma_1$ and $ \gamma_i$ satisfy the hypotheses of Theorem \ref{thmReductGeneral} so that $\gamma_1$ and $\gamma_i$ reduce to the same monomial contradicting that $f$ was reduced.

Thus $c \ne 0$ and $ \alpha \gamma_1$
reduces to the same monomial as $\beta \gamma_j$,
for some $j\in [k]$.
But $g$ is reduced, so $j\ne 1$.
Set $i_1 = 1$ and $i_2=j$.
Repeating the process,
$\alpha \gamma_{i_2}$ reduces to the same monomial as $\beta \gamma_{i_3}$,
for some $i_3 \not = i_2$,
and in general at the $p$th step
$\alpha \gamma_{i_p}$ reduces to the same monomial as $\beta \gamma_{i_{p+1}}$
with $i_p \not = i_{p+1}$.
We have finitely many monomials in  $f$, whereas 
$p$ can be arbitrarily large:
hence $i_p=i_q$ for some $q < p$. 

Now
 $\alpha^{p-q} \gamma_{i_q} \cdots \gamma_{i_p-1}$ reduces to the same monomial as
$\beta^{p-q} \gamma_{i_q} \cdots \gamma_{i_p-1}$.
It follows that $\alpha^{p-q} \gamma_{i_q} \cdots \gamma_{i_p-1}$ and $\beta^{p-q} \gamma_{i_q} \cdots \gamma_{i_p-1}$  satisfy the hypotheses of Theorem \ref{thmReductGeneral}, so do $\alpha^{p-q}$ and $\beta^{p-q}$,
and so do $\alpha$ and $\beta$,
hence $\alpha $ and $\beta$ reduce to the same monomial yielding a contradiction.

Now let $\Bbbk$ be an arbitrary field,
 denote its algebraic closure with $\overline \Bbbk$ 
 and the corresponding polynomial ring with $\overline R= \overline \Bbbk[x_a :~ \in \N]$.
By the first part of the proof the ideal   $\It{s}{t}_\Sw \overline R$ is prime.
But $\It{s}{t}_\Sw$ is the contraction of $\It{s}{t}_\Sw \overline R$, and since 
the extension
$R \subseteq \overline R$ is  faithfully flat we conclude that $\It{s}{t}_\Sw$ is a prime ideal.
\end{proof}

\section{Structure of the ideal $\I{s}{t}$}\label{secI}

By contrast to the ideal $\It{s}{t}$ investigated in the previous section, 
the structure of $\I{s}{t}$ is much less accessible. 
For instance, finding a Gr\"obner basis or the expression of the Hilbert function for this ideal seems out of our reach.
We are going to study the primary decomposition of $\I{s}{t}$ since  this ideal is not prime if $n \geq 3$.
This is in fact an easy consequence of Lemma \ref{lemProductMonomial}: 
choosing $a=(1,\ldots,1)$, $b=(r_1,\ldots,r_n)$ and $c_i=(1, r_2, \ldots, r_{i+1}, 1, \ldots, 1)$ for $i\in [n-2]$ we have 
$x_{c_1} \cdots x_{c_{n-2}} f_{1,a,b}
\in \I{s}{t}$
while 
$x_{c_1} \cdots x_{c_{n-2}} 
\notin \I{s}{t},  f_{1,a,b}
\notin \I{s}{t}$.

We use the next lemma to establish a  representation of the ideal $\It{s}{t}$ as colon  of~$\I{s}{t}$.

\begin{lem}\label{lemMinorAGreaterB}
Let $a,b\in \N$, $i \in [n]$ and consider the generalized minor $g=f_{i,a,b}$.
There exist  $a',b' \in \N$ such that $g=\pm f_{i,a',b'}$ and  $a'_j \geq b'_j$
for each $j \in [s]$.
\end{lem}
\begin{proof}
Without loss of generality we have  $\|a\sno \geq \|b\sno$.
Assume first  $i > s$.
Then we may simply choose $a'=\overline a$, $b'= \overline b$ and 
it is straightforward that $f_{i,a,b}=f_{i,a',b'}$.
Since $a'$, $b'$ are in normal form and 
 $\|a'\sno= \|a\sno \geq\|b\sno= \|b'\sno$, we obtain   $a'_j \geq b'_j$
for each $j \in [s]$.

Suppose now  $i \leq s$.
We may assume without loss of generality   $\lt(g)=x_a x_b$:
if this were not the case, we could consider the minor $f_{i,\sw(i,a,b),\sw(i,b,a)}=-f_{i,a,b}$ which satisfies $ \lt(f_{i,\sw(i,a,b),\sw(i,b,a)})= x_{\sw(i,a,b)} x_{i,\sw(i,b,a)}$.
In particular $a$ has the greatest value of $\|\cdot\sno$ among the four indices involved in $g$, 
so $\|a\sno \geq \|s(i,b,a)\sno=\|b\sno+a_i-b_i$ 
 and hence $\|a\sno-a_i \geq \|b\sno-b_i$.

Set $a'_i=a_i, b'_i=b_i$ and choose numbers $a'_j, b'_j\in [r_j]$, 
with $j \in [s]\setminus\{ i\}$,
such that $b'_j \leq a'_j$ and $\|a'\sno=\| a\sno$, $\|b'\sno=\|b\sno$ 
(this choice is possible  because $\|a\sno-a_i \geq \|b\sno-b_i$).
Set also $a'_j=a_j$ and $b'_j=b_j$ for $j=s+1,\ldots n$. 
We  have $\|\sw(i,b,a)\sno=\|b\sno-b_i+a_i=\|\sw(i,b',a')\sno$ and 
$\|\sw(i,a,b)\sno=\|a\sno -a_i+b_i=\|\sw(i,a',b')\sno$
because we did not change $a_i$, $b_i$ and $\|a\sno, \|b\sno$.
Moreover, $a',b', \sw(i,a',b'), \sw(i,b',a')$ are  $s$-equivalent to $a,b, \sw(i,a,b), \sw(i,b,a)$, respectively.
It follows that $f_{i, a,b}=f_{i,a',b'}$ and we have the desired inequality $a'_j \geq b'_j$ for each $j \in [s]$.
\end{proof}

\begin{prop}\label{propColonIt}
Let $\mathbf{p}= \prod x_a$ denote the product of all the (distinct) variables in $R$. Then
\begin{enumerate}
\item $\It{s}{t}=\I{s}{t}:\mathbf{p} = \I{s}{t}:\mathbf{p}^{\infty}$;
\item $\It{s}{t}$ is the smallest  prime ideal that contains $\I{s}{t}$
and that contains no monomials;
\item $\It{s}{t}$ is a primary component of $\I{s}{t}$.
\end{enumerate}
\end{prop}

\begin{proof}
We prove first that each generator $f_{i,a,b}$ of $ \It{s}{t}$ belongs to $\I{s}{t}:\mathbf{p} $.
If $\di(a,b) \le 1$ or $a_i=b_i$
then $f_{i,a,b} = 0$.
So suppose  $\di(a,b) \ge 2$,  $a_i\ne b_i$  and for example  $\|a\sno \geq \|b\sno$.
By Lemma \ref{lemMinorAGreaterB} we may assume, up to changing the sign of the generator, that $a_j \geq b_j$ for each $ j\in [s]$.

We
construct a path
 $c_0 = a, c_1, \ldots, c_l \in \N$
such that $ \di(c_{j-1},c_j)=1$ for all $j \in [l]$,
$c_{j-1},c_j$ do not differ in component $i$,
$\di(c_l,b) = 2$ 
and  no two $c_j, c_h$ are $s$-equivalent.
In order to do this, let 
$
K=\big\{k_1<k_2<\cdots < k_{l+1}\big\}=\big\{k	\in [n]\setminus\{i\}\, \big|\, a_k\ne b_k\big\},
$
so that $l= \# K -1= \di(a,b)-2$.
Set $c_0=a$ and  $c_j=\sw(\{k_1,\ldots,k_j\},a,b)$.
Essentially we switch the components in which $a$ and $b$ differ one at a time in increasing order, avoiding the component $i$.
By construction, no two indices are $s$-equivalent.
The path satisfies the desired properties and in particular all the variables 
$x_{c_j}$ are distinct.
Now we can apply Lemma \ref{lemProductMonomial} and it follows that
$x_{c_1} \cdots x_{c_l} f_{i,a,b} \in \I{s}{t}$.
Since the $x_{c_j}$ are all distinct,
we get that $\mathbf{p}$ is a multiple of
$x_{c_1} \cdots x_{c_l}$,
and this proves that $ \It{s}{t} \subseteq \I{s}{t}:\mathbf{p} $. 

The inclusion  $\I{s}{t}: \mathbf{p}\subseteq \I{s}{t}: \mathbf{p}^{\infty}$ holds in general.
For the final inclusion, given $g\in \I{s}{t}:\mathbf{p}^m$ with $m \in \NN$, we have that $g\mathbf{p}^m \in \I{s}{t}\subseteq \It{s}{t}$ and, as $\It{s}{t}$ is prime and contains no monomials, 
we get $g \in \It{s}{t}$.
So (1) is proved.

For (2), let $P$ be a  prime ideal that contains $\I{s}{t}$
and that contains no monomials.
By the proof of (1), $ \It{s}{t} = \I{s}{t} : \mathbf{p}
\subseteq P : \mathbf{p} = P$, and assertion
(2) follows.
Assertion (3) follows by (1) and the fact that $\It{s}{t}$ is prime.
\end{proof}

Now we  
extend the primeness result of Proposition \ref{propPrimeSingle} to a more general class of ideals.

\begin{prop}\label{propPrimeGeneral}
Let $\Sw$ be an $(s,t)$-switchable set.
The ideals $\It{s}{t}_{\Sw}$ and $\P{s}{t}_{\Sw}$ are prime.
\end{prop}

\begin{proof}
Let $\Sw = \Sw_1 \cup \cdots \cup \Sw_r$ be the decomposition of $\Sw$ into its connected components;
we have
$\It{s}{t}_{\Sw}= \sum_{i=1}^r \It{s}{t}_{\Sw_i}$.

Denote with $\overline \Bbbk$ the algebraic closure of $\Bbbk$
and  with $\overline R$ the polynomial ring $\overline \Bbbk[x_a : a \in \N]$: the extension
$R \subseteq \overline R$ is  faithfully flat.
Proposition \ref{propPrimeSingle} applies also to $\overline R $ and hence $\It{s}{t}_{\Sw_i} \overline R$ is  a prime ideal for each $i$.
A well-known result states that 
the sum of two prime ideals of a polynomial ring is  prime
if the generators of the two  ideals
involve disjoint sets of variables and the underlying field is algebraically closed field.
It follows that the ideals
$$
\It{s}{t}_{\Sw} \overline R=\sum_{i=1}^r \It{s}{t}_{\Sw_i} \overline R 
\quad \text{and} \quad 
\P{s}{t}_{\Sw} \overline R = \It{s}{t}_{\Sw}  + \Var{s}{t}_\Sw  \overline R
$$ 
are prime.
Since the generators of these two ideals are in $R$
and $R \subseteq \overline R$ is  faithfully flat,
we get that
$\It{s}{t}_{\Sw}$ and $\P{s}{t}_\Sw$ are 
contractions of the prime ideals
$\It{s}{t}_{\Sw}\overline R$ and $\P{s}{t}_\Sw \overline R$
and 
hence  they are prime.
\end{proof}

An $(s,t)$-switchable set  $\mathcal{S}\subseteq \mathcal{N}$ is  {\bf maximal $\bf (s,t)$-switchable} if
for all $(s,t)$-switchable sets $\mathcal{T}$  containing $\Sw$ properly, the ideals
$\P{s}{t}_\Sw$ and $\P{s}{t}_\Tw$ are incomparable.

\begin{thm}\label{thmMinimalPrimes}
The set of minimal primes over $\I{s}{t}$ 
consists of all the ideals of the form $\P{s}{t}_\Sw$, where $\Sw$ is a maximal $(s,t)$-switchable set.
\end{thm}

\begin{proof}
We have already seen in Proposition \ref{propColonIt} that  $\It{s}{t}$ is a minimal prime of $\I{s}{t}$.
Notice that $\It{s}{t}$  is of the form of the statement since $\It{s}{t} = \P{s}{t}_\N$,
and $\N$ is trivially a maximal $(s,t)$-switchable set.

Let $P$ be an arbitrary prime ideal minimal over $\I{s}{t}$.
We assume that $P \ne \It{s}{t}$ and thus it contains some variables, otherwise $\It{s}{t}\subseteq P$ by Proposition \ref{propColonIt} (2).
Define the set
$
\Sw=\{a \in \N : x_a \not \in P\}.
$
We have $\Sw \ne \emptyset$, otherwise $P$ is the maximal homogeneous ideal which properly contains $\It{s}{t}$.
We prove now that $\Sw$ is $(s,t)$-switchable.

Given $a, b \in \Sw$ such that $\di(a,b) = 2$ and $a_i\neq b_i$ for some $i\in[t]$, we prove that $\sw(i,b,a), \sw(i,a,b) \in \Sw$.
Since $P$ contains $\I{s}{t}$ and $i\in[t]$, $P$ contains the slice minor $f_{i,a,b}  = x_a
x_b - x_{\sw(i,b,a)} x_{\sw(i,a,b)}$.
Since $a, b \in \Sw$,
then $x_a x_b \not \in P$,
so that 
$x_{\sw(i,b,a)} x_{\sw(i,a,b)} \not \in P$,
which implies $x_{\sw(i,b,a)}, x_{\sw(i,a,b)} \not \in P$,
and hence
$\sw(i,b,a), \sw(i,a,b) \in \Sw$.
Thus $\Sw$ satisfies property (1) in Definition \ref{definSTgood}.
By definition of $\Sw$  it is straightforward that $\Sw$ is 
closed under $s$-equivalence, proving that $\Sw$ is $(s,t)$-switchable.
Now $\P{s}{t}_\Sw$ is a prime ideal containing $\I{s}{t}$ by  Lemma \ref{lemPcontainsI}.

Let us prove that $\P{s}{t}_\Sw \subseteq P$.
By definition of $\Sw$ we have $\Var{s}{t}_\Sw \subseteq P$.
Let $f_{i,a,b} \in \It{s}{t}_\Sw$,
with $i\in[t]$ and $a$ and $b$ connected in $\Sw$:
by Lemma \ref{lemPathAB}
there exist elements $c_0=a, \ldots, c_l\in \Sw$ such that 
$\di(c_l,b) = 2$,
$\di(c_{j-1},c_j)=1$ 
and 
$c_{j-1},c_j$ do not differ in component $i$
for all $j \in [l]$.
Then we can apply Lemma \ref{lemProductMonomial} and we get
$x_{c_1} \cdots x_{c_l} f_{i,a,b} \in \It{s}{t}_\Sw \subseteq P$
and, since $x_{c_j} \not \in P$ for all $j \in [l]$,
it follows that $f_{i,a,b} \in P$.
Thus $\I{s}{t} \subseteq \P{s}{t}_\Sw \subseteq P$ and by minimality  $\P{s}{t}_\Sw = P$.

Finally,
let $\Tw$ be an $(s,t)$-switchable set properly containing $\Sw$; then $\Var{s}{t}_\Tw \subsetneq \Var{s}{t}_\Sw$.
By  Lemma \ref{lemPcontainsI}
$\P{s}{t}_\Tw$ contains $\I{s}{t}$.
This, combined with the fact that $P = \P{s}{t}_\Sw$ is minimal over $\I{s}{t}$,
implies that $\It{s}{t}_\Tw\not\subseteq \P{s}{t}_\Sw$.  Therefore 
$\P{s}{t}_\Sw$ and $\P{s}{t}_\Tw$ are incomparable, and
 $\Sw$ is a maximal $(s,t)$-switchable set.

Now we prove the converse.
Let $\Sw$ be a maximal $(s,t)$-switchable set;
we know that 
$\P{s}{t}_\Sw$ is a prime ideal containing $\I{s}{t}$.
Take a minimal prime  $P$ of $\I{s}{t}$ contained in $\P{s}{t}_\Sw$.
By the first part of the proof we have
$P = \P{s}{t}_\Tw$ for a maximal $(s,t)$-switchable set $\Tw$.
By  $\P{s}{t}_\Tw \subseteq \P{s}{t}_\Sw$ it follows that $\Var{s}{t}_\Tw \subseteq \Var{s}{t}_\Sw$ and hence $\Sw \subseteq \Tw$.
But $\Sw$ is maximal, and  $\P{s}{t}_\Sw$ and $\P{s}{t}_\Tw$ are comparable:
necessarily we have $\Sw = \Tw$ and $\P{s}{t}_\Sw$ is thus a minimal prime.
\end{proof}

\section{Minimal primes for large values of $s$}\label{secNandN-1}

In this section we are going to specialize $s$ and make Theorem \ref{thmMinimalPrimes} more explicit.
We start with a result on $(s,t)$-switchable sets that actually holds for any $s \geq 2$,  but will  be particularly  helpful when $s=n,n-1$.

\begin{lem}\label{lemVarsGoodSets}
Let $\Sw$ be an $(s,t)$-switchable set with $s \geq 2$. 
Assume that $\Sw$ contains an element $a$ such that $(a_1,\ldots,a_s)\ne (1,\ldots,1), (r_1,\ldots,r_s)$.
Then all the indices $b\in \N$ such that $(b_{s+1}, \ldots, b_{n})=(a_{s+1}, \ldots, a_{n})$ belong to $\Sw$.
\end{lem}
\begin{proof}
Since $\Sw$ is closed under $s$-equivalence,
we may assume without loss of generality that both $a$ and $b$ are in normal form.
Notice that our hypothesis implies $a_1>1$ and $a_s<r_s$.

We proceed by induction on $\delta(a,b)= \|a-b\sno$.
If $\delta(a,b)=0$ then $a=b$ and the conclusion holds trivially.
Assume now  $\delta(a,b)>0$ and for example  $\|a\sno>\|b\sno$.
Consider the following index $a'\in \N$:
$$
a'_1=a_1-1,\qquad a'_s=a_s+1,\qquad a'_i =a_i\quad \text{for }\, i \in [n]\setminus \{1,s\}.
$$
We have that $a'$ is $s$-equivalent to $a$ and thus $a'\in \Sw$.
Since $\di(a,a')=2$ it follows that also $c'=\sw(1,a,a')\in \Sw$.
Let $c$ be the normal form of $c'$,
 we have that $c\in \Sw$ and $\|b\sno \leq \|c\sno <\|a\sno$, so that $\delta(b,c)<\delta(a,b)$ and the conclusion follows by induction.
\end{proof}

In the next part of this  section we study the minimal primes of $\I{n}{t}$ and therefore we fix $s=n$.
We can restrict our focus to $n\geq 3$:
if $n=2$ then we have a Hankel matrix, for which $\I{2}{t}=\It{2}{t}$ is prime and well-understood.

\begin{prop}\label{propGoodSetsSN}
Assume that $n\geq 3$. Then the $(n,t)$-switchable sets are exactly 
$$
\Sw_1=\emptyset,\, \Sw_2=\{(1,\ldots,1)\},\, \Sw_3=\{(r_1, \ldots, r_n)\},
$$
$$
\Sw_4= \{(1,\ldots,1),(r_1, \ldots, r_n)\},\, \Sw_5=\N.
$$
\end{prop}
\begin{proof}
First we observe that $\Sw_4$ satisfies property (1) in Definition \ref{definSTgood} if and only if $n\geq 3$, as the distance between the two indices is $n$. 
The other four sets satisfy trivially that property.
Moreover, all  five sets are closed under $n$-equivalence, therefore they are  $(n,t)$-switchable.
These are the only $(n,t)$-switchable sets: if $\Sw \ne \N$ is a proper $(n,t)$-switchable subset, then by Lemma \ref{lemVarsGoodSets} $\Sw \subseteq \Sw_4$ (otherwise $\Sw=\N$) and the conclusion follows.
\end{proof}

\begin{thm}\label{thmMinimalPrimesSN}
If $n\geq 3$ then the ideal $\I{n}{t}$ has exactly two minimal primes: the ideal $\It{n}{t}$ and the monomial prime
$
\P{n}{t}_{\Sw_4}=\big( x_a\,|\, a \ne(1,\ldots,1),(r_1, \ldots, r_n) \big).
$
\end{thm}
\begin{proof}
We want to apply Theorem \ref{thmMinimalPrimes}.
We already know that $\It{n}{t}=\P{n}{t}_\N $ is a minimal prime.
By Proposition \ref{propGoodSetsSN} we only have to look for maximal $(n,t)$-switchable sets in the list 
$\{ \Sw_1, \Sw_2, \Sw_3,\Sw_4\}.
$
The set $\Sw_4$ is maximal: 
indeed the only $(n,t)$-switchable set containing $\Sw_4$ is $\N$, 
and $\P{n}{t}_\N$ and $\P{n}{t}_{\Sw_4}$ are not comparable as  $f_{1,(1,\ldots,1),(r_1, \ldots, r_n)} \notin \P{n}{t}_{\Sw_4}$.
On the other hand, we have $\P{s}{t}_{\Sw_4} \subseteq \P{s}{t}_{\Sw_i}$ for each $i=1,2,3$ and thus these 3 sets are not maximal.
\end{proof}

For the first minimal prime Proposition \ref{propColonIt} gives us an expression as colon ideal.
The next two lemmas will lead to a similar  expression for the other minimal prime ideal.

\begin{lem}\label{lemMidpoint}
Let $a, b \in \N$ be two indices such that
$\| a\nno -\|b\nno=2$.
Let $c \in \N$ be  an index such that $\|c\nno=\|a\nno-1=\|b\nno+1$.
Then  $x_a x_b = x_c^2 \pmod{\I{n}{t}}$.
\end{lem}
\begin{proof}
Since $\sum_{i=1}^n (a_i-b_i)=2$ there exist  $ a'_i,b'_i\in [r_i]$ such that 
$$
a'_1-b'_1=1,\quad a'_2-b'_2=1,\quad a'_i=b'_i \quad  \text{ for } i\geq 3,$$
$$
 \|a'\nno=\|a\nno, \quad \|b'\nno= \|b \nno.
$$
Thus $a',b'$ are $n$-equivalent to $a,b$, respectively, and  $x_a=x_{a'},x_b=x_{b'}$.
We have $\di(a',b')=2$ and thus the minor $f_{1,a',b'}=x_{a'}x_{b'}-x_{\sw(1,a',b')}x_{\sw(1,b',a')}$  belongs to $\I{n}{t}$.
But this minor proves the conclusion, because $\| \sw(1,a',b')\nno=\|a\nno-1$, $\| \sw(1,b',a')\nno=\|a\nno-1$ and hence these two indices define the same variable as $c$.
\end{proof}

\begin{lem}\label{lemExponent}
Let $a \in \N\setminus \{(1,\ldots,1),(r_1, \ldots, r_n)\}$. 
There exists  $p \in \NN $ such that $x_a^p$ is equivalent $ \pmod{ \I{n}{t}}$ to a monomial divisible by the product $\mathbf{p}$ of all variables.
\end{lem}
\begin{proof}
We are going to prove first the following claim: for any $b\in \N$ there exists a positive integer $m_b$ such that $x_a^{m_b}\equiv x_b \alpha \pmod{ \I{n}{t}}$, for some monomial $\alpha$.
We proceed by induction on the quantity $\gamma(a,b)=\big| \|a\nno - \|b\nno \big|$.
If $\gamma(a,b)=0$, then $a,b$ are $n$-equivalent and we can just choose $m_b=1, \alpha=1$.
Assume now $\delta(a,b)>0$ and for example $\|a\nno > \|b\nno $.
Take $c\in \N$ such that $\|c\nno = \|b\nno+1 $:
we have $\gamma(b,c)=1$ and $\gamma(a,c)<\gamma(a,b)$.
In particular, by induction, there exists a positive integer $m_c$ such that 
$x_a^{m_c}\equiv x_c \beta \pmod{ \I{n}{t}}$ for some monomial $\beta$.
Since $a <(r_1,\ldots,r_n)$ and $\|c\nno \leq \|a\nno$, we have that $c<(r_1,\ldots,r_n)$ and hence there exists $d\in \N$ such that $\|d\nno = \|c\nno+1=\|b\nno+2 $.
By Lemma \ref{lemMidpoint} we get that 
 $x_b x_d = x_c^2 \pmod{\I{n}{t}}$, and consequently 
 $
 x_a^{2m_c}\equiv x_c^2 \beta^2\equiv  x_b x_d \beta^2\pmod{ \I{n}{t}}
 $
 so that the claim follows setting $m_b=2m_c$ and $\alpha=x_d\beta^2$.
 
Now to obtain the statement of the proposition it is sufficient to choose $p \geq \sum_{b \in \N} m_b$, where $m_b$ is the positive integer provided by the claim for each $b\in \N$.
\end{proof}

Let $\rad(\cdot)$ denote the radical of an ideal.
 
\begin{prop}
Let  $\P{n}{t}_{\Sw_4}$ be the monomial minimal prime of $\I{n}{t}$ as in Theorem \ref{thmMinimalPrimesSN}.
We have $ \P{n}{t}_{\Sw_4}= \rad\big({\I{n}{t} : f_{1,(1,\ldots,1),(r_1, \ldots, r_n)}}\big)$.
\end{prop}
\begin{proof}
Let us prove  that $x_a \in \rad\big({\I{n}{t} : f_{1,(1,\ldots,1),(r_1, \ldots, r_n)}}\big)$ for any variable $x_a \in \P{n}{t}_{\Sw_4}$.
Consider the  indices
$
c_0=(1,\ldots,1), c_j=(1,\ldots,1,r_{n-j+1},\ldots,r_n)$ for $j\in [n-3]$.
These indices form a path between $(1,\ldots,1)$ and $(1,1,r_3,\ldots,r_n)$ and  they do not differ in the first component.
By Lemma \ref{lemProductMonomial}
$x_{c_1} x_{c_2} \cdots x_{c_{n-3}} f_{1,(1,\ldots,1),(r_1, \ldots, r_n)} \in \I{n}{t}$.
Moreover, the variables $x_{c_j}$ are all distinct because 
the values $\|c_j\nno$ are strictly increasing.
It follows that $x_{a_1} x_{a_2} \cdots x_{a_{n-3}}$ divides the product $\mathbf{p}$ of all variables, and thus $\mathbf{p}\in \I{n}{t}: f_{1,(1,\ldots,1),(r_1, \ldots, r_n)}$.
By Lemma \ref{lemExponent} there exists a positive integer $p$ such that $x_a^p$ is equivalent $\pmod{\I{n}{t}}$ to a monomial divisible by $\mathbf{p}$,
so it follows that $x_a^p \in \I{n}{t}: f_{1,(1,\ldots,1),(r_1, \ldots, r_n)}$
 and $x_a \in \rad\big(\I{n}{t}: f_{1,(1,\ldots,1),(r_1, \ldots, r_n)}\big)$.

Conversely,
let $g \in \rad\big(\I{n}{t} : f_{1,(1,\ldots,1),(r_1, \ldots, r_n)}\big)$.
We have $g^p f_{1,(1,\ldots,1),(r_1, \ldots, r_n)} \in \I{n}{t} \subseteq \P{n}{t}_{\Sw_4}$ for some positive integer $p$;
but $f_{1,(1,\ldots,1),(r_1, \ldots, r_n)} \notin \P{n}{t}_{\Sw_4}$ and $\P{n}{t}_{\Sw_4}$ is prime,
so that $g^p \in \P{n}{t}_{\Sw_4}$ and hence $g \in \P{n}{t}_{\Sw_4}$, 
proving that  $\P{n}{t}_{\Sw_4}$ contains
$
\rad\big({\I{n}{t} : f_{1,(1,\ldots,1),(r_1, \ldots, r_n)}} \big).
$
\end{proof}

Unlike the case of the minimal prime $\It{n}{t}$, the $\P{n}{t}_{\Sw_4}$-component is not prime in general.

\begin{ex}
Let $\N=[3] \times [3] \times [3]$, $s=t=n=3$,
then the $\P{n}{t}_{\Sw_4}$-component  is
$
\big( x_{(3, 3, 1)},\,  x_{(3, 2, 1)},\, x_{(3, 1, 1)},\, x_{(3, 3, 2)}^2,\,  x_{(2, 1, 1)}x_{(3, 3, 2)},\, x_{(2, 1, 1)}^2         \big) \subsetneq \P{n}{t}_{\Sw_4}.
$
\end{ex}

We have evidence to believe that if $s=n$ there are no embedded primes, and that the primary components are exactly the two colon ideals mentioned above (without the radical).

\begin{con}\label{conjPrimary}
The  primary decomposition of $\I{n}{t}$ is 
$$
\I{n}{t}= \It{n}{t} \cap \big( \I{n}{t} : f_{1,(1,\ldots,1),(r_1, \ldots, r_n)} \big).
$$
\end{con}

Now we turn to the study of the minimal primes of $\I{n-1}{t}$
and hence we fix $s=n-1$ for the remainder of this section.
First we find all the $(n-1,t)$-switchable sets.
Given  ${B}\subseteq [r_n]$ we define the following subset of $\N$:
$$
\Sw_{{B}}=[r_1]\times \cdots \times [r_{n-1}]\times B.
$$
We deal separately with the cases $n=3$ and $n\geq 4$,
which yield different decompositions.

\begin{prop}\label{propSwitchable4}
Assume that $n\geq 4$.
The $(n-1,t)$-switchable sets consist of two classes: 
\begin{enumerate}
\item
all the subsets of the form
$
\Sw_{B}
$
where ${B}$ is any subset of $[r_n]$;

\item all the subsets $\Sw\subseteq \N$ such that for all $ a \in \Sw$ we have either $(a_1, \ldots, a_{n-1})=(1,\ldots, 1),$ or
$(a_1, \ldots, a_{n-1})=(r_1,\ldots, r_{n-1})$. 
\end{enumerate}
\end{prop}
\begin{proof}
First we notice that the two types of sets of the statement are $(n-1,t)$-switchable.
This is clear for the first type, whereas for the second type it is true because  there are no $a,b\in \Sw$ such that $\di(a,b)=2$, and it is trivially closed under $(n-1)$-equivalence.

Now we prove that these are the only ones.
Let $\Sw$ be an $(n-1,t)$-switchable set.
Assume that $\Sw$ is not of type (2) in the statement, that is to say, $\Sw$ contains an index $b$ such that $(b_1, \ldots, b_{n-1})\ne(1,\ldots, 1),$ 
$(r_1,\ldots, r_{n-1})$;
we claim that $\Sw$ is of the form (1) in the statement.
If $\Sw$ is not of that form, then there are two indices $a\in \Sw, c\notin \Sw$ such that $a_n=c_n$.
In particular, we must have $(a_1, \ldots, a_{n-1})\in  \{(1,\ldots, 1),(r_1,\ldots, r_{n-1})\}$, otherwise by Lemma \ref{lemVarsGoodSets} we would have $c \in \Sw$.
For example  $(a_1, \ldots, a_{n-1})= (1,\ldots, 1)$.
We also must have $b_n\ne c_n$, otherwise $c \in \Sw$ again by Lemma \ref{lemVarsGoodSets}.
Consider the index $e=(2,1,\ldots, 1, b_n)$: since $e_n=b_n$ we have $e	\in \Sw$ by Lemma \ref{lemVarsGoodSets}.
Since $a_n=c_n\ne b_n$ we get  $\di(a,e)=2$ and hence $f=\sw(1,a,e)\in \Sw$.
But $f_n=a_n=c_n$, and $(f_1,\ldots, f_{n-1})\ne (1,\ldots, 1),(r_1,\ldots, r_{n-1})$ because $n\geq 3$ and $f_2=1$.
It follows that  $c\in \Sw$ by Lemma \ref{lemVarsGoodSets}, which is a contradiction.
We conclude that $\Sw$ is of type (1) in the statement, with $B= \{b_n: \, b \in \Sw\}$.
\end{proof}

We see now that if $n\geq 4$ the minimal primes of $\I{n-1}{t}$ are very similar to those of $\I{n}{t}$, except that the monomial prime is generated by all but $2r_n$ variables instead of all but $2$.

\begin{thm}\label{thmMinimalPrimes4}
If $n\geq 4$ then the ideal $\I{n-1}{t}$ has exactly two minimal primes: the ideal $\It{n-1}{t}$ and the monomial prime
$
P=\big( x_a\,|\, (a_1, \ldots, a_{n-1})\ne (1,\ldots, 1),(r_1,\ldots, r_{n-1}) \big)
$.
\end{thm}
\begin{proof}
By Theorem \ref{thmMinimalPrimes} we have to find all the maximal $(n-1,t)$-switchable sets, 
and certainly $\N$ is one of these. 
In fact, it is the only one among those of type (1) in Proposition \ref{propSwitchable4}: 
if ${B} \subsetneq [r_3]$ then it is easy to check that $\It{n-1}{t} \subsetneq \P{n-1}{t}_{\Sw_{{B}}}$.

Now consider $(n-1,t)$-switchable sets of type (2). 
There is a maximal one with respect to inclusion, namely the set $\Sw=\big\{a \in \N |\, (a_1,\ldots,a_{n-1})\in \{ (1,\ldots, 1), (r_1,\ldots , r_{n-1}) \}\big\}$.
This is indeed a maximal $(n-1,t)$-switchable set, because the only $(n-1,t)$-switchable set containing it is $\N$ and we have
$\P{n-1}{t}_\N	\not\subseteq \P{n-1}{t}_\Sw$ because $f_{1,(1,\ldots,1),(r_1,\ldots, r_n)}\notin \P{n-1}{t}_\Sw$.
None of the other $(n-1,t)$-switchable sets of the form (2) is maximal: if $\Tw$ is one of these, then $\Tw \subsetneq \Sw$ and we can easily verify  that $\P{n-1}{t}_\Sw	\subsetneq \P{n-1}{t}_\Tw$.
Therefore the two ideals $\It{n-1}{t}$ and $P=\P{n-1}{t}_\Sw$ provide the list of minimal primes of $\I{n-1}{t}$.
\end{proof}

\begin{prop}\label{propSwitchable3}
Assume that $n=3$.
The $(2,t)$-switchable sets consists of two classes: 
\begin{enumerate}

\item
all the subsets of the form
$
\Sw_{{B}},
$
with $B \subseteq [r_3]$;

\item all the subsets $\Sw\subseteq \N$ satisfying the following property:
for all $l \in [r_3]$ there is at most one $a\in \Sw$ such that $a_3=l$, 
and either $(a_1,a_2)=(1,1)$ or $(a_1,a_2)=(r_1,r_2)$.
\end{enumerate}
\end{prop}
\begin{proof}
The two types of sets of the statement are $(2,t)$-switchable.
The first type is the same as in Proposition \ref{propSwitchable4}.
For the second type,  we have $\di(a,b)\ne 2$  for all $a,b\in \Sw$ and  $\Sw$ is closed under $2$-equivalence.

Let us prove that these are the only ones.
If $\Sw$ is a $(2,t)$-switchable set and   contains an index $b$ such that $(b_1,b_2)\ne (1,1),(r_1,r_2)$, then
with the very same proof as in the case $n\geq 4$ we can conclude that $\Sw$ is of type (1).
Assume now that for every $a\in \Sw$ we have either $(a_1,a_2)=(1,1)$ or $(a_1,a_2)=(r_1,r_2)$.
If there exists  $l\in [r_3]$ such that $a=(1,1,l)\in \Sw$ and $b=(r_1,r_2,l)\in \Sw$, then $\di(a,b)=2$ but $\sw(1,a,b)=(r_1,1,l)\notin \Sw$, contradiction.
Thus $\Sw$ is  of type (2).
\end{proof}

We have seen that if $s=n\geq 3$ or $s=n-1$ and $n \geq 4$ then there are exactly two minimal primes.
On the other hand, if $n=3$ and $s=2$ then there are more minimal primes.

\begin{thm}\label{thmMinimalPrimes3}
If $n=3$, the ideal $\I{2}{t}$ has exactly the following $2^{r_3}-1$ minimal primes: 
the ideal $\It{2}{t}$ and
the monomial ideals  
$
\P{2}{t}_{\Sw^\varphi}
$
with
$$
\Sw^{\varphi}=\big\{a\in \N :\, (a_1,a_2)=\varphi(a_3) \big\}
 $$
 as $\varphi$ varies over the non-constant functions $\varphi : [r_3]\rightarrow \big\{(1,1),(r_1,r_2)\big\}$.
\end{thm}
\begin{proof}
Again, we have to find all the maximal $(2,t)$-switchable sets.
As in Theorem \ref{thmMinimalPrimes4}, $\N$ is the only maximal $(2,t)$-switchable set
 among those of type (1) in Proposition \ref{propSwitchable3}.
Now consider  sets of type (2), i.e., $\Sw$ such that for all $l \in [r_3]$ there is at most one $a\in \Sw$ such that $a_3=l$, 
and  either $(a_1,a_2)=(1,1)$ or $(a_1,a_2)=(r_1,r_2)$.

First of all, we may assume that $ \Sw$ is maximal with respect to inclusion among the sets of type (2) in Proposition \ref{propSwitchable3}, or equivalently that for all $l \in [r_3]$ there is \emph{exactly} one $a\in \Sw$ such that $a_3=l$.
Indeed if $\Sw\subsetneq \Tw$ are two $(2,t)$-switchable sets of the form (2), then $\P{2}{t}_\Tw \subsetneq \P{2}{t}_\Sw$,  because they are  equal to $\Var{2}{t}_\Tw$ and $\Var{2}{t}_\Sw$, respectively (there are no two connected indices with distance at least 2 in $\Sw$  or $\Tw$ and thus $\It{2}{t}_\Sw=\It{2}{t}_\Tw=\{0\})$.

Therefore $\Sw=\Sw^\varphi$ for some function $\varphi : [r_3]\rightarrow \{(1,1),(r_1,r_2)\}$.
If $\varphi$ is constant, then $S^\varphi$ is not a maximal $(2,t)$-switchable set: we claim that $\It{2}{t} \subseteq \P{2}{t}_{\Sw^\varphi}$.
Let $f_{i,a,b}$ be a nonzero generator of $\It{2}{t}$, so that $\di(a,b)>1$.
If $a\notin  \Sw^\varphi$ or $b\notin  \Sw^\varphi$ then at least one of $\sw(1,a,b),\sw(1,b,a)$ does not belong to $\Sw^\varphi $  and hence $f_{i,a,b} \in \P{s}{t}_{\Sw^\varphi}$.
Notice that since $\varphi$ is constant and $\di(a,b)>1$ we cannot have both $a,b\in \Sw^\varphi$.
The claim is thus proved.

So necessarily $\Sw=\Sw^\varphi$ for some non-constant function $\varphi : [r_3]\rightarrow \{(1,1),(r_1,r_2)\}$.
Let us see that the converse holds.
We prove that 
$\It{2}{t}_\Tw \not\subseteq \P{2}{t}_{\Sw^\varphi} $ for every $\Tw$ of type (1) containing $\Sw^\varphi$. 
It is sufficient to check this for $\Tw=\N$, because $\It{2}{t}_\N \subseteq \It{2}{t}_\Tw$ for these $\Tw$.
To this purpose consider, as usual, the element $f_{1,(1,1,1),(r_1,r_2, r_3)}$:
 it doesn't belong to $\P{2}{t}_{\Sw^\varphi} $ because $x_{(r_1,1,1)},x_{(1,r_2,r_3)}\in \P{2}{t}_{\Sw^\varphi}$ and $x_{(1,1,1)},x_{(r_1,r_2,r_3)}\notin \P{2}{t}_{\Sw^\varphi}$.

These monomial primes are in a one-to-one correspondence with the set of non-constant functions $\varphi : [r_3]\rightarrow \{(1,1),(r_1,r_2)\}$, and there are $2^{r_3}-2$ such functions; 
considering also the binomial minimal prime $\It{2}{t}$ we obtain exactly $2^{r_3}-1$ primes.
\end{proof}

\begin{rem}
We note that when $s<n$ there are embedded primes.
For example, consider the $2$-Hankel hypermatrix indexed by $\N=[2]\times [2] \times [3]$:
the maximal homogeneous ideal is an associated prime of $\I{2}{1}$.
\end{rem}

\begin{rem}\label{remLargeValueS}
We have seen in Theorems \ref{thmMinimalPrimesSN}, \ref{thmMinimalPrimes4} and \ref{thmMinimalPrimes3} that  all minimal primes over $\I{s}{t}$ other than $\It{s}{t}$ are monomial if $s\geq n-1$.
This is no longer true  for smaller values of $s$.
For instance, take the $3$-Hankel hypermatrix indexed by $ \N=[2]\times [2]\times [2]\times [2]\times [2]$, i.e.,
 with $n=5$ and $s=3$.
 The set $ \Sw = \{ a \in \N : \, a_1=a_2=a_3\}$ 
 is maximal $(3,5)$-switchable,
 $\It{3}{5}_\Sw \not \subseteq \mathrm{Var}^{\langle 3,5\rangle}_\Sw$ and therefore  $\P{3}{5}_\Sw$ is not monomial.
\end{rem}

\section{Watanabe's theorem and projective varieties}\label{secGeometry}

Given two ideals $I$ and $I'$ living in two isomorphic rings $R$ and $R'$, when
is there an isomorphism  $\Psi:R\rightarrow R'$ such that $\Psi(I)=I'$?
Invariants of  ideals,
such as the Hilbert function, 
allow us to solve this problem in one direction (when the answer is negative).
 In this section we address this question in the other direction, for ideals of the form $\It{s}{t}$.

The idea is to generalize to hypermatrices the following result proved by J. Watanabe (in the more general setting of $r\times r$ minors).
Let $H$ and $H'$ be an  $m\times n$  and an $m'\times n'$ Hankel matrix, respectively, and let $I_2(H)$, $I_2(H')$ denote the ideals generated by all the $2 \times 2$ minors:
    if $m+n=m'+n'$ then $I_2(H) \cong I_2(H')$ (cf. \cite[Theorem 1]{Wa}).
    Looking at $s$-Hankel hypermatrices now, choose two sets of  integers   $\{n,s,t,r_1,\ldots,r_n\}$ and $\{n',s',t',r'_1,\ldots,r'_{n'}\}$ with the usual limitations $s,t \in [n]$, $s',t'\in [n']$, $r_i,r'_i\geq 2$.
These  parameters determine two  hypermatrices $M$ and $M'$ indexed  in
$\mathcal{N}=[r_1]\times \cdots \times [r_{n}]$ and  $\mathcal{N}'=[r'_1]\times \cdots\times [r'_{n'}]$, respectively.
Then we have the corresponding polynomial rings $R$ and $R'$, and the ideals $\I{s}{t}$ and $\It{s}{t}$ which we denote with  $I$ and $\widetilde{I}$ for $R$ and 
$I'$ and $\widetilde{I}'$ for $R'$.
In order to address the question above, we need to have $R \cong R'$.
By Remark \ref{remDimensionR},  if $n-s=n'-s'$, $ 
\sum_{i=1}^s r_i-s = \sum_{i=1}^{s'} r'_i-s$ and $r_{n-i+1}=r'_{n'-i+1}$ for $i\in [n-s]$
then $\dim R= \dim R'$ and hence the two rings are isomorphic.
Moreover, by Theorem \ref{corHilbert}, the Hilbert functions  of $\widetilde{I}$ and $\widetilde{I}'$ are the same if one of the two following conditions holds:
either  $t > s$ and $t-s=t'-s'$, or $t\leq s$ and $t'\leq s'$.
We see that in this case $\widetilde{I}$ and $\widetilde{I}'$ are in fact isomorphic.

\begin{thm}\label{thmGeneralizationWatanabe}
Let  $\{n,s,t,r_1,\ldots,r_n\}$ and $\{n',s',t',r'_1,\ldots,r'_{n'}\}$  be two sets of parameters as above.
If the parameters satisfy the following  conditions
\begin{itemize}
\item $n-s=n'-s'$;

\item $ 
\sum_{i=1}^s r_i-s = \sum_{i=1}^{s'} r'_i-s'$;

\item $ r_{n-i+1}=r'_{n'-i+1} $ for $i \in [n-s]$;
\end{itemize}
and one of the two conditions
\begin{itemize}
\item $s<t$ and $t-s=t'-s'$;

\item $t\leq s$ and $t'\leq s'$;
\end{itemize}
then there is an isomorphism $\Psi: R \rightarrow R'$ that maps $\widetilde{I}$ to $\widetilde{I}'$.

\end{thm}
\begin{proof}
Let  $\{x_a, a \in \N\}$ be the variables of $R$ and  $\{x'_b, b \in \N'\}$ the variables of $R'$.
Notice that if $s=1$ the two hypermatrices $M$ and  $M'$ have the same sizes and the theorem doesn't really say anything. Therefore, assume $s \geq 2$.

We can define an isomorphism $\Psi: R \rightarrow R'$ assigning a bijection between the variables of the two rings.
We map a variable $x_a$ to $x'_{\psi(a)}$, where $\psi(a) \in \N'$ is an index 
such that $\|\psi(a)\|_{\langle s' \rangle}= 	\|a\sno$ and the last $n-s$ components of $ \psi(a)$ are the same as those of $a$. 
The assumptions on the parameters and the $s$-equivalence and $s'$-equivalence guarantee that this is a well-defined bijection between  the two sets of variables.

Now this isomorphism preserves the quantity and multisets of properties (1), (2a)-(2b) for a monomial $\alpha \in R$
from Theorem \ref{thmReductGeneral}.
Since Theorem \ref{thmReductGeneral} characterizes the membership for binomials and ideals $\widetilde{I}$, $\widetilde{I}'$,
we can deduce that  $\Psi$ maps binomials of $\widetilde{I}$ to binomials of $\widetilde{I}'$.
The argument works clearly in the other direction too, so that $\Psi (\widetilde{I})= \widetilde{I}'$.
\end{proof}

\begin{rem}
In general, even though the hypotheses of Theorem \ref{thmGeneralizationWatanabe} are satisfied, $I$ and $I'$ need not be isomorphic.
For example, if $\N=[3]\times [3]\times[3]$, $s=t=n=3$ then 
the  Hilbert polynomial of  $I$ is  $P(d)=9d-2$,
whereas if $\N'=[4]\times [2]\times[3]$, $s'=t'=n'=3$ the  Hilbert polynomial of  $I'$ is $P'(d)=7d$.
\end{rem}

Besides generalising Watanabe's theorem, Theorem \ref{thmGeneralizationWatanabe}  allows us to relate the ideal $\It{s}{t}$ to some classical determinantal varieties.
A good reference on the topic is \cite{Har}.

\begin{cor}\label{corCurve}
An ideal of the form $\It{n}{t}$, i.e. with $s=n$, defines a rational normal curve.
\end{cor}
\begin{proof}
Recall that equations for the rational normal curve in $\mathbb{P}^m$ are given by the $2 \times 2$ minors of a Hankel matrix of size $2 \times m$
whose entries are the $m+1$ homogeneous coordinates.
Assume that $\It{n}{t}$ is associated to the set of indices $\mathcal{N}=[r_1]\times \cdots \times [r_{n}]$.
We define another set of parameters setting $n'=s'=t'=2, r'_1=2, r'_2=\sum_{i=1}^n r_i-n$
(we have $r'_2 \geq 2$ because $r_i\geq 2$ for each $i\in [n]$).
Now the two sets of parameters satisfy the hypotheses of Theorem \ref{thmGeneralizationWatanabe}, 
and the second one defines a $2\times r'_2$ Hankel matrix.
The theorem follows with $m = \sum_{i=1}^n r_i-n$.
\end{proof}

Now we consider a generalization of the rational normal curve to higher dimensions.
Let $\sigma_1\leq \sigma_2 \leq  \cdots  \leq \sigma_l$ be a non-decreasing sequence in $\NN$ and set $m= \sum_{i=1}^l \sigma_i+l-1$.
We can find complementary $\sigma_i$-dimensional linear subspaces $\Lambda_i  \subseteq \mathbb{P}^m$ and rational normal curves $\mathcal{C}_i \subseteq \Lambda_i$ in each.
Choose biregular maps$  $ $\varphi_i : \mathcal{C}_1 \rightarrow \mathcal{C}_i$ and let 
$$
\Scroll({\sigma_1, \ldots, \sigma_l}) = \bigcup_{P\in \mathcal{C}_1} \overline{P, \varphi_2(P), \ldots,  \varphi_l(P)}
$$
where $\overline{P_1, \ldots,  P_l}$ denotes the linear subspace of $\mathbb{P}^m$  spanned by points  ${P_1, \ldots,  P_l}$.
The set 
$\Scroll({\sigma_1, \ldots, \sigma_l})$ is called a {\bf rational normal scroll of dimension $\bf l$}   and it is actually a projective variety: 
its defining ideal is generated by the minors of the following  matrix (cf.~\cite{Har})
$$ \Sigma = \left(
\begin{matrix}
Z_{1,1} & Z_{1,2}  & \cdots & Z_{1,\sigma_1} \\
Z_{1,2} & Z_{1,3}  & \cdots & Z_{1,\sigma_1+1}
\end{matrix}
\right|
\left.
\begin{matrix}
\, \cdots \, \\
 \, \cdots \,
\end{matrix}
\right|
\left.
\begin{matrix}
Z_{l,1} & Z_{l,2}  & \cdots & Z_{l,\sigma_l} \\
Z_{l,2} & Z_{l,3}  & \cdots & Z_{l,\sigma_l+1}
\end{matrix}
\right)
$$
i.e., a matrix consisting of $l$ blocks of sizes $2\times \sigma_i$, with each block a Hankel matrix.
We denote by $I_2(\Sigma)$  this ideal, which lives in the  polynomial ring 
$
S=\Bbbk\big[Z_{i,j}: i\in[l],\,1\leq j \leq \sigma_i\big].$

Consider now the class of rational normal scrolls such that all the defining integers are equal, that is to say $\Scroll(\sigma, \ldots, \sigma)$ with $\sigma\in \mathbb{N}_+$ repeated $l$ times.
We want to view the corresponding ring $S
$
as a ring of the form $R$ used so far.
We can map the variables of  the polynomial ring $R$ , with set of indices $N=[2]\times [\sigma]\times [l]$ and parameters  $n=3$, $s=2$, to those of the polynomial ring $S$ above.
The correspondence between the two sets of variables is given by 
$
\psi(x_{(a_1,a_2,a_3)})= Z_{a_1+a_2-1,a_3}.
$
This map is well-defined and bijective because of the $2$-equivalence and corresponding
identification among variables $x_a$ of $R$.
Therefore we have defined an isomorphism $\Psi:R\rightarrow S$.
\begin{cor}\label{corScroll}
An ideal of the form $\It{n-1}{t}$, i.e. with $s=n-1$,  defines a rational normal scroll if $n \geq 3$.
\end{cor}
\begin{proof}
We consider at first the particular case $n=3,s=2,t=1$  and $\mathcal{N}=[2]\times [\sigma]\times[l]$ for some $\sigma, l \in \NN$.
Since $t=1$, by Discussion \ref{discFlattening} the ideal $\It{2}{1}$ can be realized as the  determinantal ideal of the flattening of $M$ with respect to the first component, i.e., the following matrix:
$$\left(
\begin{matrix}
x_{(1,1,1)} & x_{(1,2,1)}  & \cdots & x_{(1,\sigma,1)} \\
x_{(2,1,1)} & x_{(2,2,1)}  & \cdots & x_{(2,\sigma,1)}
\end{matrix}
\right|
\left.
\begin{matrix}
\, \cdots \, \\
 \, \cdots \,
\end{matrix}
\right|
\left.
\begin{matrix}
x_{(1,1,l)} & x_{(1,2,l)}  & \cdots & x_{(1,\sigma,l)} \\
x_{(2,1,l)} & x_{(2,2,l)}  & \cdots & x_{(2,\sigma,l)}
\end{matrix}
\right).
$$ 
Thus,
 denoting this matrix by $M'$  and its $2 \times 2$ determinantal ideal by $I_2(M')$, 
 we have $\It{2}{1}=I_2(M')$.
But the isomormphism $\Psi:R\rightarrow S$ maps the matrix $M'$  to the matrix $\Sigma$ and therefore the  ideal $I_2(M')$ to $I_2(\Sigma)$, and the conclusion follows.

Now consider the general case. 
Notice first that $t=n$ and $t=n-1$ produce the same ideal by Remark \ref{remTminusOne}, thus we may assume  $t \leq n-1$.
We want to apply Theorem \ref{thmGeneralizationWatanabe}, and to this purpose we choose the following parameters:
$
n'=3,\, s'=2,\, t'=1,\, r'_1=2,\, r'_2=\sigma= \sum_{i=1}^{n-1}r_i-n+1,\, r'_3=l=r_n.
$
With this choice the ring $R$ is isomorphic to the polynomial ring $R'$ associated to the set of indices $\N'=[2]\times[\sigma]\times[l]$ and $s'=2$, and the isomorphism maps our ideal $\It{n-1}{t}$ to the ideal $\It{2}{1}$ of $R'$.
But this ideal defines a rational normal scroll, as seen in the first part of the proof.
\end{proof}

\section*{Acknowledgments}
This paper is an outcome of the author's master thesis written under the supervision of Irena Swanson.
We thank her for the great hospitality at Reed College and for many helpful discussions.
We are also thankful to the  referee for improving the organization of the content.

\vspace*{1cm}
\end{document}